\documentclass[11pt,reqno]{amsart}

\usepackage[bottom=3.5cm]{geometry}%[margin=1in]
\usepackage{pstricks}
\usepackage{float, graphicx}
\usepackage{latexsym}
\usepackage{mathrsfs}
\usepackage{amsmath, amsthm, amssymb,amsfonts}
\usepackage{epsfig}
\usepackage{verbatim}
\usepackage{url}
\usepackage[colorlinks, bookmarks=true]{hyperref}
\usepackage{cleveref}
\usepackage{graphicx}
\usepackage{enumerate}
\usepackage{bm}
\usepackage{dsfont}
\usepackage{stmaryrd}
\usepackage{enumitem}
\usepackage{mathtools}
\usepackage{comment}
\usepackage{color}
\usepackage{xcolor}
\usepackage[font={scriptsize}]{caption}  %%%% Pour changer la taille de caption
\usepackage{centernot}
\newtheorem{theorem}{Theorem}[section]
\newtheorem{lemma}[theorem]{Lemma}

\newtheorem{proposition}[theorem]{Proposition}
\newtheorem{corollary}[theorem]{Corollary}

\newtheorem{assumption}{Assumption}

\theoremstyle{definition}   % text in "Roman"
\newtheorem{remark}[theorem]{Remark}

\newcommand{\beq}{\begin{equation}}
\newcommand{\eeq}{\end{equation}}

\usepackage{pgf,tikz}
\usetikzlibrary{arrows}
\usetikzlibrary{arrows.meta}
\usetikzlibrary{calc}
\usetikzlibrary{shapes}
\usetikzlibrary{backgrounds}
\usetikzlibrary{decorations.markings}
\usetikzlibrary{decorations}
\usetikzlibrary{decorations.pathreplacing}
\tikzset{individu/.style={fill,thick,circle}}
\def\e{{\mathbb E}}
\def\p{{\mathbb P}}

\def\z{{\mathbb Z}}
\def\N{{\mathbb N}}
\def\T{{\mathbb T}}

\def\ee{e}
\def\d{\, \mathrm{d}}
\newcommand{\dist}{\mathrm{dist}}

\newcommand{\HH}{ \mathcal{H}}

\allowdisplaybreaks

\numberwithin{equation}{section}
\usepackage{environ}

\begin{document}
\makeatletter
\def\@settitle{\begin{center}%
  \baselineskip14\p@\relax
    %\bfseries
    \normalfont\LARGE%<- NEW
  \@title
  \end{center}%
}
\makeatother
% avoid uppercasing title
\title[Random walk on dynamical percolation: critical vs.\ supercritical]{\bf Random walk on dynamical percolation in Euclidean lattices: separating critical and supercritical regimes}
 
\let\MakeUppercase\relax % this disables uppercasing authors

\author[C. Gu]{\bf \large Chenlin Gu} 
\address{Yau Mathematical Sciences Center, Tsinghua University, Beijing, China.}
\email{gclmath@tsinghua.edu.cn}

\author[J. Jiang]{Jianping Jiang} 
\address{Yau Mathematical Sciences Center, Tsinghua University, Beijing, China.}
\email{jianpingjiang@tsinghua.edu.cn}

\author[Y. Peres]{Yuval Peres}
\address{Beijing Institute of Mathematical Sciences and Applications, Beijing, China.}
\email{yperes@gmail.com}

\author[Z. Shi]{Zhan Shi}
\address{Academy of Mathematics and Systems Science, Chinese Academy of Sciences, Beijing, China.}
\email{shizhan@amss.ac.cn}

\author[H. Wu]{Hao Wu}
\address{Yau Mathematical Sciences Center, Tsinghua University, and Beijing Institute of Mathematical Sciences and Applications, Beijing, China.}
\email{hao.wu.proba@gmail.com}

\author[F. Yang]{Fan Yang}
\address{Yau Mathematical Sciences Center, Tsinghua University, and Beijing Institute of Mathematical Sciences and Applications, Beijing, China.}
\email{fyangmath@tsinghua.edu.cn}

\begin{abstract}
We study the random walk on dynamical percolation of $\mathbb{Z}^d$ (resp., the two-dimensional triangular lattice $\mathcal{T}$), where each edge (resp., each site) can be either open or closed, refreshing its status at rate $\mu\in (0,1/\ee]$. The random walk moves along open edges in $\mathbb{Z}^d$ (resp., open sites in $\mathcal{T}$) at rate $1$. For the critical regime $p=p_c$, we prove the following two results: on $\mathcal{T}$, the mean squared displacement of the random walk from $0$ to $t$ is at most $O(t\mu^{5/132-\epsilon})$ for any $\epsilon>0$; on $\mathbb{Z}^d$ with $d\geq 11$, the corresponding upper bound for the mean squared displacement is $O(t \mu^{1/2}\log(1/\mu))$. For the supercritical regime $p>p_c$, we prove that the mean squared displacement on $\mathbb{Z}^d$ is at least $ct$ for some $c=c(d)>0$ that does not depend on $\mu$.
\end{abstract}

\maketitle
 
%{
%  \hypersetup{linkcolor=black}
%  \tableofcontents
%}

\begin{figure}
    \centering
    \includegraphics[width=0.25\textwidth]{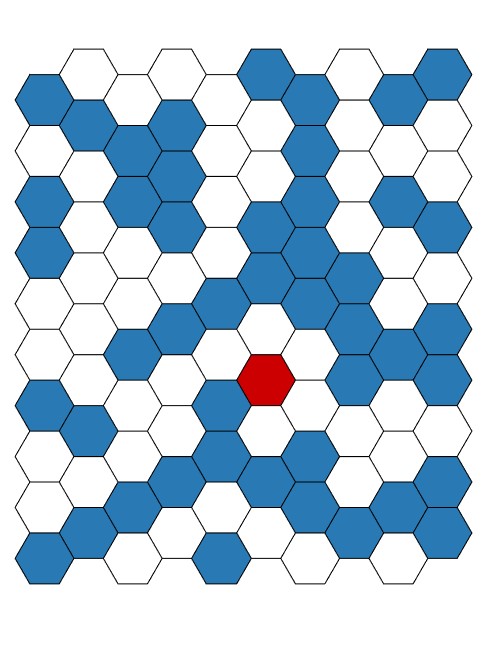}\hspace{0.5cm}
    \includegraphics[width=0.25\textwidth]{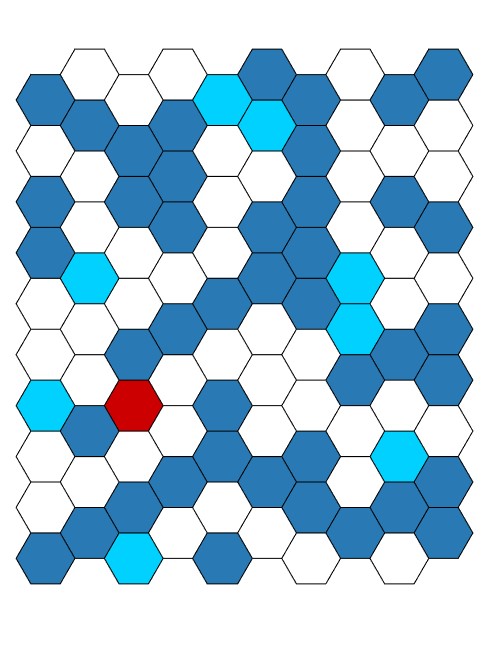}\hspace{0.5cm}
    \includegraphics[width=0.25\textwidth]{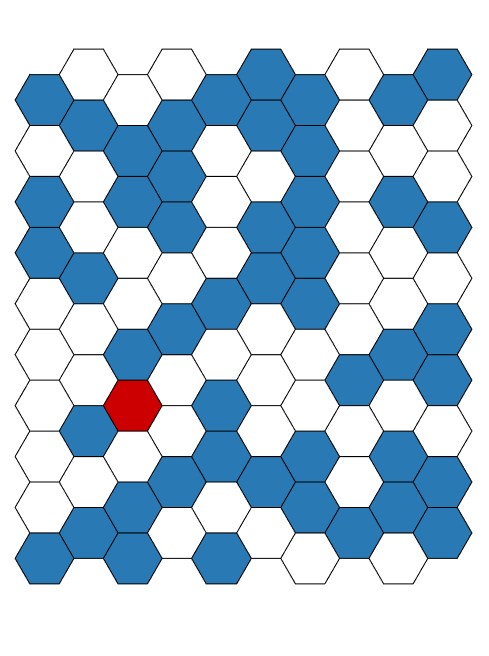}
    \caption{Two snapshots (left and right panels) of random walk on dynamical percolation on the faces of the  Hexagonal lattice which is dual to the triangular lattice .   Open sites are in white, closed sites are in dark blue, and the moving particle is in red. Sites that refresh between the snapshots are indicated in the middle panel in light blue. }\label{fig.T}
\end{figure}

\section{Introduction}
\label{s:intro}
Let $G=(V,E)$ be an infinite graph with the vertex set $V$ and the edge set $E$. For any $x,y\in V$, let $\dist(x,y)$ denote the graph distance between $x$ and $y$ in $G$. We study random walk on dynamical percolation on $G$. Each edge refreshes independently at rate $\mu\leq 1/\ee$ and transitions to an open state with probability $p$, or to a closed state with probability $1-p$. The random walk $(X_t)$ moves at rate $1$. When its clock rings, the walk selects one of the adjacent edges with equal probability; it jumps to the neighboring site if the selected edge is open and stays still if the edge is closed. Peres, Stauffer, and Steif first introduced this model in \cite{peres-stauffer-steif}.

%In this paper, we mainly focus on $(\mathbb{Z}^d,E(\mathbb{Z}^d))$.  Let $p_c(\mathbb{Z}^d)$ be the critical probability for bond percolation on $\mathbb{Z}^d$. 
We mainly focus on the Euclidean lattices $G=(\mathbb Z^d,E(\mathbb{Z}^d))$, where $E(\mathbb{Z}^d)$ represents the set of edges connecting nearest-neighbor vertices in $\mathbb{Z}^d$. Let \smash{$\eta_t\in\{0,1\}^{E(\mathbb{Z}^d)}$} denote the configuration of open edges at time $t$, and let $p_c=p_c(\mathbb{Z}^d)$ be the critical probability for bond percolation on $\mathbb{Z}^d$. Consider a random walk starting from $X_0=0$ and assume that the initial law for dynamical percolation follows the product Bernoulli measure \smash{$\pi_p:=\mathrm{Ber}(p)^{E(\mathbb{Z}^d)}$}, where $\mathrm{Ber}(p)$ is the Bernoulli measure which takes $1$ (or open) with probability $p$ and $0$ (or closed) with probability $1-p$. Notice that  $\pi_p$ is the invariant measure for the process $(\eta_t)_{t \geq 0}$. 

In the subcritical regime $p\in(0,p_c)$, it was proved in \cite[Corllary 1.6]{peres-stauffer-steif} that $\mathbb{E} [\dist^2(0,X_t)] \leq C[(\mu t)\vee 1]$ for all $t\geq 0$. In \cite[Corollary 1.10]{peres-stauffer-steif}, a general upper bound $\mathbb{E} [\dist^2(0,X_t)] \leq Ct$ was established for all $p\in[0,1]$, $\mu\in[0,1]$, and $t\geq0$. 
Furthermore, it can be deduced from \cite[Theorem~1.2]{peres-stauffer-steif} that $\mathbb{E} [\dist^2(0,X_t)] \geq c\mu t$ for all $p\in (0,p_c)$, $\mu\in(0,1]$, and $t\geq 0$. 
One may extend the lower bound $\mu t$ to all phases $p\in (0,1]$ by the method in Peres, Sousi, and Steif \cite{PSS18}. 
In this paper, we study the mean squared displacement in more details. More precisely, we will prove a matching lower bound (see Theorem \ref{t:lb} below) for the supercritical regime. In addition, we will establish an upper bound for the critical regime in terms of the one-arm exponent and correlation-length exponent. Consequently, we find that the mean squared displacement behaves differently in critical and supercritical regimes. 

We first state our main results at criticality. For $d=2$, we need to modify our model slightly since the relevant critical exponents for bond percolation on $\mathbb{Z}^2$ have not been proved rigorously. Consider the two-dimensional triangular lattice $\mathcal{T}:=(V,E)$, where $V:=\{x+ye^{i\pi/3}: x,y \in \mathbb{Z}\}$ and $E:=\{\{x,y\}: x, y\in V, \|x-y\|_2=1\}$. In this model, each vertex refreshes independently at rate $\mu\leq 1/\ee$ and transitions to an open state with critical probability $p_c(\mathcal T)=1/2$, or to a closed state with probability $1/2$. The random walk $(X_t)$ moves at rate $1$. When its clock rings, the walk selects one of the $6$ adjacent vertices with equal probability; it jumps to the vertex if it is open and stays still if it is closed. We state our main result about the random walk on dynamical percolation on $\mathcal{T}$ in the following theorem. 
%We denote by $\mathrm{Ber}(p)$ the Bernoulli measure which takes $1$ (or open) with probability $p$ and $0$ (or closed) with probability $1-p$.

\begin{theorem}\label{t:ubtri}
	Consider the random walk on the dynamical percolation on $\mathcal{T}$ at criticality $p=1/2$, starting from $X_0=0$ and the initial law $\pi_{1/2}=\mathrm{Ber}(1/2)^{V}$ for the dynamical percolation. For each $\epsilon\in(0,5/132)$, there exists a constant $C_{\ref{t:ubtri}}=C_{\ref{t:ubtri}}(\epsilon)\in(0,\infty)$ (independent of $\mu$) such that 
 \begin{equation}\label{eqn::conclusionT}
		\mathbb{E} [\dist^2(0,X_t)] \leq C_{\ref{t:ubtri}}t \mu^{\frac{5}{132}-\epsilon},\quad  \forall \mu\in(0,1/\ee],
	\end{equation}
 for all sufficiently large $t$ (depending on $\mu$). See \eqref{eq:triub} below for a sharper statement. 
\end{theorem}

\begin{remark}
We will derive an upper bound for $\mathbb{E} [\dist^2(0,X_t)]$ in Proposition \ref{prop:ublowd} in terms of the one-arm exponent $\alpha_1$ and the correlation-length exponent $\nu$ for (static) percolation. For site percolation on $\mathcal{T}$, it has been proved rigorously that $\nu=4/3$ by Kesten \cite{Kes87} and Smirnov and Werner~\cite{SW01}, and $\alpha_1=5/48$ by Lawler, Schramm and Werner \cite{LSW02} (up to an $o(1)$ correction in the exponent, accounting for the $\epsilon$ in the theorem). These exponents essentially lead to the estimate~\eqref{eqn::conclusionT}.
\end{remark}

We also obtain a corresponding result for the $\z^2$ case.
\begin{proposition}\label{prop:Z^2}
Consider the random walk on the dynamical percolation on $\mathbb{Z}^2$ at criticality $p=1/2$, starting from $X_0=0$ and the initial law $\pi_{1/2}=\mathrm{Ber}(1/2)^{E(\mathbb Z^2)}$ for the dynamical percolation. There exist constants $\delta\in (0,1)$ and $C_{\ref{prop:Z^2}}\in(0,\infty)$ (independent of $\mu$) such that
  \begin{equation}
		\mathbb{E} [\dist^2(0,X_t)] \leq C_{\ref{prop:Z^2}}t\mu^{\delta} ,\quad \forall \mu\in(0,1/\ee],
	\end{equation}
 for all sufficiently large $t$ (depending on $\mu$). 
 See Corollary \ref{cor:Z^2} below for a sharper statement. 
\end{proposition}

Concerning the random walk on high-dimensional dynamical percolation at criticality, we focus on $\mathbb{Z}^d$ with $d\ge 11$, where   mean-field behavior has been rigorously established in the literature.  Our argument applies to all lattices where the one-arm exponent $\alpha_1$ and the correlation length exponent  $\nu$ take their mean field values $2$ and $1/2$, respectively.

%same result is expected to hold for all $d \ge 6$.
\begin{theorem}
    \label{t:ubhigh}
    Fix $d\geq 11$. Consider the random walk on dynamical percolation on $\mathbb{Z}^d$ at criticality $p=p_c(\mathbb{Z}^d)$, starting from $X_0=0$ and the initial law $\pi_p=\mathrm{Ber}(p)^{E(\mathbb Z^d)}$ for the dynamical percolation. There exists a constant $C_{\ref{t:ubhigh}}=C_{\ref{t:ubhigh}}(d)\in(0,\infty)$ (independent of $\mu$) such that 
    \begin{equation}
        \e\left[\dist^2(0,X_t)\right] \le C_{\ref{t:ubhigh}} t \mu^{1/2}\log(1/\mu),\quad \forall t \geq \mu^{-1/2},~  \mu\in(0,1/\ee]. 
    \end{equation}
\end{theorem}

In the supercritical regime $p\in(p_c(\mathbb{Z}^d),1]$,  one may deduce from \cite{PSS20} a lower bound of the form $\e[\dist^2(0,X_t)] \geq t(\log t)^{-c}$ for some constant $c>0$ and a certain range of $t$. In the following theorem, we present a tight lower bound that matches the upper bound $Ct$ proved in~\cite{peres-stauffer-steif}.

%For the random walk on the dynamical percolation on $\mathbb{Z}^d$ in the supercritical regime $p>p_c(\mathbb{Z}^d)$, we already mentioned that an upper bound for the mean squared displacement is proved in~\cite{peres-stauffer-steif}. On the other hand, one may deduce a lower bound (for $\e[\dist^2(0,X_t)]$) of the form $t(\log t)^{-c}$ for some $c\in(0,\infty)$ from~\cite{PSS20}. Apparently, this lower bound is not optimal. In the following theorem, we obtain a tight lower bound which matches the upper bound in~\cite{peres-stauffer-steif}.

\begin{theorem}\label{t:lb}
   Fix $d\geq 2$. Consider the random walk on the dynamical percolation on $\mathbb{Z}^d$ with $p\in (p_c(\mathbb{Z}^d),1]$, starting from $X_0=0$ and the initial law $\pi_p=\mathrm{Ber}(p)^{E(\mathbb Z^d)}$ for the dynamical percolation. There exists a constant $c_{\ref{t:lb}}=c_{\ref{t:lb}}(d,p)\in(0,\infty)$ (independent of $\mu$) such that 
   \begin{equation}\label{eqn::supercritical_lower}
   	\e\left[\dist^2(0,X_t)\right] \geq c_{\ref{t:lb}} t,\quad \forall t\geq 0,~ \mu\in(0,1/\ee]. 
   \end{equation}
\end{theorem}

\begin{comment}
\begin{remark}\label{rem::supercritical_sharp}
Combining~\cite[Corollary~1.10]{peres-stauffer-steif} and~\eqref{eqn::supercritical_lower},  we get that for each $p>p_c(\mathbb{Z}^d)$ there exists $C_{\ref{t:lb}}\in(0,\infty)$ such that
\[c_{\ref{t:lb}} t\le \e\left[\dist^2(0,X_t)\right]\le C_{\ref{t:lb}} t, ~\forall t\geq0,~\forall\mu\in(0,1/\ee].\]
\end{remark}
\end{comment}

%\begin{remark}
Although we present Theorem \ref{t:lb} specifically for $\mathbb{Z}^d$, the same proof can be applied to other Euclidean lattices, such as the triangular lattice $\mathcal{T}$. It is clear that Theorem \ref{t:lb} separates the supercritical behavior from the critical behavior, as stated in Theorems \ref{t:ubtri} and \ref{t:ubhigh} and Proposition~\ref{prop:Z^2}, for all small $\mu$. 
%\end{remark}

\medskip
\noindent{\bf Notations.} 
In this paper, we will use the set of natural numbers $\N=\{1,2,3,\ldots\}$. Throughout the paper, the symbols $c$ and $C$ will represent positive constants that vary from place to place. These constants may depend on the dimension $d$ and percolation probability $p$, but not on $\mu$. 
For any numbers $a$ and $b$, we use the notation $a\lesssim b$ or $a=O(b)$ to mean that $|a|\le C|b|$ for a constant $C\in(0,\infty)$ that does not depend on $\mu$. We write $a\asymp b$ if $a\lesssim b$ and $b\lesssim a$.
%\footnote{We write $f\lesssim g$ if $f/g$ is bounded from above by finite constant that does not depend on $\mu$; we write $f\asymp g$ if $f\lesssim g$ and $g\lesssim f$.}
\subsection{Diffusion constant}
Theorem 3.1 of \cite{peres-stauffer-steif} and uniform integrability imply that the following diffusion constant for the random walk on dynamical percolation is well-defined:
	\begin{equation}\label{eq:diffusionconstant}	\sigma^2(d,p,\mu)=\lim_{t\rightarrow\infty}\frac{1}{t}{\mathbb{E}\left[\dist^2(0,X_t)\right] }.
	\end{equation}
 (Note that the usual definition of diffusion constant uses the Euclidean distance but in the above definition we use the graph distance.)
Our results, along with those established in \cite{peres-stauffer-steif,PSS18,PSS20}, provide the following bounds on the diffusion constant, which we summarize here for the reader's convenience. 
 
%We have the following summary.
\begin{itemize}
    \item Subcritical regime $p<p_c(\z^d)$: 
    \begin{equation}
        \sigma^2(d, p, \mu)\asymp \mu. 
    \end{equation}
    This is proved in~\cite{peres-stauffer-steif} as mentioned earlier.
    \item Supercritical regime $p>p_c(\z^d)$: 
    \begin{equation}
        \sigma^2(d,p,\mu)\asymp 1. 
    \end{equation}
    The upper bound is proved in~\cite{peres-stauffer-steif} and the lower bound is given by Theorem~\ref{t:lb}. 
    \item Critical regime $p=p_c(\z^d)$: 
    \begin{itemize}
        \item When $d=2$, there exists a constant $\delta\in (0,1)$ such that 
    \begin{equation}
        \mu\lesssim \sigma^2(d,p,\mu)\lesssim \mu^{\delta}. 
    \end{equation}
    The upper bound is proved in Proposition~\ref{prop:Z^2} and the lower bound can be derived from~\cite{PSS18}.
    \item When $d\ge 11$, 
    \begin{equation}
        \mu\lesssim \sigma^2(d,p,\mu)\lesssim \mu^{1/2}\log(1/\mu).
    \end{equation}
    The upper bound is proved in Theorem~\ref{t:ubhigh} and the lower bound can be derived from~\cite{PSS18}.
    \end{itemize} 
\end{itemize}

\subsection{Overview and structure}

We sketch the main ideas of the proof and give an overview of the paper's structure. 

For the critical case (Section \ref{sec:cri}), we focus on the open cluster of the origin, which consists of all edges that are open at some time during $[0,t]$. We classify this cluster according to its diameter and bound the corresponding mean squared displacement using the forward/backward martingale decomposition for reversible Markov processes (see Propositions \ref{prop:tailGeneral} and \ref{prop:tlogt} below). Then, the problem boils down to estimating the probability of the cluster having a certain diameter. We obtain such estimates by studying the stability of the one-arm probability in the near-critical regime. For $d=2$, this follows from the scaling relations established by Kesten \cite{Kes87} (see also Nolin \cite{Nol08}), the one-arm exponent by Lawler, Schramm and Werner \cite{LSW02}, and the correlation-length exponent by Smirnov and Werner ~\cite{SW01}. For $d\geq 11$, we rely on the two-point function asymptotics by Fitzner and van der Hofstad~\cite{FvdH17}, and the intrinsic and extrinsic one-arm exponents by Kozma and Nachmias \cite{KN09,KN11}.

For the supercritical case (Section \ref{sec:sup}), our overall strategy is similar to that in \cite{PSS20}. That is, we first fix the environment to obtain a time-inhomogeneous Markov chain. Then, we consider the evolving sets of this Markov chain introduced by Morris and Peres \cite{MP05} and employ the Diaconis-Fill coupling \cite{DF90} between the random walk and the Doob transform of the evolving sets.  
The problem reduces to proving that the evolving set has the right size. The main challenge is to get rid of various logarithmic dependencies in \cite{PSS20}. For example, in the definition of good times, we require that the intersection of the Doob transform of the evolving set with the infinite cluster occupies a positive fraction of the set, as opposed to a fraction that vanishes asymptotically as in \cite{PSS20}. 
In our proof, we mainly work with the probability measure when the initial law for percolation is stationary.  An advantage of this measure is that the random walk is stationary, which avoids the study of the hitting time to giant components as in \cite{PSS20}. Consequently, we can establish that with positive probability, the fraction of good times is strictly positive. Applying a result of Pete~\cite{Pet08} on the isoperimetric profile of a set in the infinite cluster (see Corollary \ref{cor:isoperimetric} below, which also improves upon the corresponding result in \cite{PSS20}) to the set of good times, we obtain a good drift for the size of the evolving set along these times.

\subsection{Related works} 

We provide a review of some relevant works on random walks in evolving random environments. We specifically focus on Euclidean lattices and refer the reader to our companion paper~\cite{GJPSWY24} for   related works on general underlying graphs.

We first mention some works that specifically focus on random walks on dynamical percolation. The concept of dynamical percolation on arbitrary graphs was introduced by H{\"a}ggstr{\"o}m, Peres, and Steif \cite{OYS97}. Later, Peres, Stauffer, and Steif \cite{peres-stauffer-steif} introduced the model of random walk on dynamical percolation, and several results have been established in \cite{peres-stauffer-steif,PSS18,PSS20} as mentioned earlier. Hermon and Sousi \cite{HS20} extended the setting to general underlying graphs and proved a comparison principle between the random walk on dynamical percolation and the random walk on the underlying graph. Lelli and Stauffer \cite{LS24} investigated the mixing time of random walk on a dynamical random cluster model on $\mathbb Z^d$, where edges switch at rate $\mu$ between open and closed, following a Glauber dynamics. Recently, Andres, Gantert, Schmid, and Sousi \cite{AGSS23} considered the biased random walk on dynamical percolation on $\mathbb Z^d$ and established several results such as the law of large numbers and an invariance principle for the random walk.

%We refer to our companion paper~\cite{GJPSWY24} for more history and related works on general settings. 

The random conductance model is also an active direction (see Biskup \cite{Bis11} for a survey), where many results have been generalized to the dynamical setting. Some early works \cite{BMP97,BZ06,DKL08,DL09} by Boldrighini, Minlos, Pellegrinotti, Bandyopadhyay, Zeitouni, Dolgopyat, Keller, and Liverani considered random walks in dynamical Markovian random environments and established the corresponding invariance principles. In a similar setting, Redig and V\"ollering \cite{RV13}
obtained the strong ergodicity properties for the environment as seen from the walker. % thereby transferring the rate of mixing of the environment to the rate of mixing of the environment process.
Andres \cite{Andres14} derived the quenched invariance principle for a stationary ergodic dynamical model with uniform ellipticity. Later, Andres, Chiarini, Deuschel, and Slowik \cite{ACDS18} relaxed the uniform ellipticity to certain moment conditions, while den Hollander, dos Santos, and Sidoravicius~\cite{HSS13} established a law of large numbers for a class of non-elliptic random walks on $\mathbb Z^d$ in dynamic random environments. Biskup \cite{Biskup:2019EJP} proved an invariance principle for one-dimensional random walks among dynamical random conductances. Biskup, Rodriguez, and Pan \cite{BR18,BP03} considered generalizations to degenerate dynamical environments, where the speed may vanish for some time interval. 
Dolgopyat and Liverani \cite{DL08} studied random walks in environments with deterministic, but strongly chaotic evolutions.
Blondel \cite{Blondel15} proved diffusion properties for random walks in environments governed by a kinetically constrained spin model at equilibrium. 
In recent work \cite{HH2022}, Halberstam and Hutchcroft studied collisions of two conditionally independent random walks in a dynamical environment on $\mathbb Z^2$. 
Furthermore, there are also several works studying random walks in dynamical environments where the jump rate is associated with vertices by Avena, Blondel, Faggionato, den Hollander, and Redig \cite{ABF18,AHR11}, as well as random walks among exclusion processes by Avena, den Hollander,  Redig, dos Santos, and V\"ollering~\cite{Avena2012,AHR09,ASV13}.

\section{Critical case: proofs of Theorems \ref{t:ubtri} and \ref{t:ubhigh}}
\label{sec:cri}

In this section, we will study the random walk on dynamical percolation at criticality. Our goal is to prove Theorems \ref{t:ubtri} and \ref{t:ubhigh} and \Cref{prop:Z^2}. For $r\in \mathbb{N}$ and the lattice $\mathcal T$, we define
\begin{align*}
    B_r &:=\{x+ye^{i\pi/3}:  |x| \leq r, |y| \leq r\} \cap\mathcal{T},\\
    \partial B_r &:=\{x\in B_r: x \text{ has a nearest neighbor in }\mathcal{T}\setminus B_r\}.
\end{align*}
With a slight abuse of notation, we also define for $r\in \mathbb N$ and the lattice $\mathbb Z^d$ that
\begin{align*}
    B_r &:=[-r,r]^d \cap\mathbb{Z}^d, \\
    \partial B_r &:=\{x\in B_r: x \text{ has a nearest neighbor in }\mathbb{Z}^d\setminus B_r\}.
\end{align*}
Note $B_r$ and $\partial B_r$ depend on the underlying lattice; their meanings will be clear from the context.

\subsection{Mean squared displacement: a general upper bound}
We first prove a general upper bound for the mean squared displacement under the following assumption on the one-arm probability. 
\begin{assumption}\label{a:1arm}
There exist constants $\tilde{\alpha}_0, \tilde{\alpha}_1, \tilde{\nu}, C_0, C_1\in(0,\infty)$ such that
	\begin{equation}
		\frac{C_0}{r^{\tilde{\alpha}_0}} \leq \mathbb{P}_p(0\longleftrightarrow \partial B_r)\leq \frac{C_1}{r^{\tilde{\alpha}_1}}, \quad \forall r\geq 1, ~ p\in[p_c(G),p_c(G)+r^{-1/\tilde{\nu}}],
	\end{equation}
 where $G$ is either $\mathbb{Z}^d$ or $\mathcal{T}$. Here, $\p_p$ denotes the probability measure of the Bernoulli-$p$ bond (resp.,~site) percolation on $\mathbb Z^d$ (resp.,~$\mathcal T$), where every edge (resp.,~site) is independently open with probability $p$.
\end{assumption}

Let us elaborate on this assumption.
\begin{itemize}
\item It is believed that $\tilde{\alpha}_0=\tilde{\alpha}_1$. However, the exact value of $\tilde{\alpha}_0$ is not important (while $\tilde{\alpha}_1$ is crucial) for our applications. The existence of such $\tilde{\alpha}_0$ can be found for example in \cite{BE22}.
\item Assumption \ref{a:1arm} holds for $G=\mathcal{T}$ and $G=\mathbb{Z}^2$, as stated in  Lemmas~\ref{lem:onearm2d_triangle} and~\ref{lem:onearm2d_Z} below. In these two cases, the optimal values for $\tilde{\alpha}_0=\tilde{\alpha}_1$ and $\tilde{\nu}$ are believed to be the one-arm exponent $\alpha_1=5/48$ and the correlation-length exponent $\nu=4/3$. However, these exponents have been rigorously proved only for $G=\mathcal{T}$; see Lemma~\ref{lem:onearm2d_triangle}. 
\item Assumption~\ref{a:1arm} holds for $\mathbb{Z}^d$ with $d\ge 11$, as stated in Lemma~\ref{lem:onearmhd} below. In this case, we have $\tilde{\alpha}_0=\tilde{\alpha}_1=2$ and $\tilde{\nu}=1/2$, which are the one-arm and correlation-length exponents,  respectively.

%In this case, we have the one-arm exponent $\tilde{\alpha}_1=2$ and the correlation-length exponent $\tilde{\nu}=1/2$. 
\item Assumption \ref{a:1arm}
is believed to be true for $\mathbb{Z}^d$ with $3\le  d \le 10$. However, for the upper bound, even the problem of whether $\lim_{r\to\infty}\mathbb{P}_{p_c}(0\longleftrightarrow \partial B_r)=0$ or not remains open in this case. 
\end{itemize}
Under \Cref{a:1arm}, we can prove the following proposition provided that $\tilde{\alpha}_1\in(0,2)$ and $\tilde{\nu} \geq 1/2$.

\begin{proposition}\label{prop:ublowd}
Let $G=\mathcal{T}$ or $\mathbb{Z}^d$ with $d\geq 2$. Suppose Assumption \ref{a:1arm} holds with $\tilde{\alpha}_1\in (0,2)$ and $\tilde{\nu}\geq 1/2$. Consider the random walk on the dynamical percolation on $G$ at criticality $p=p_c(G)$, starting from $X_0=0$ and the initial law $\pi_p=\mathrm{Ber}(p)^{V}$ if $G=\mathcal{T}$ (or $\mathrm{Ber}(p)^{E(\mathbb{Z}^d)}$ if $G=\mathbb{Z}^d$) for the  dynamical percolation. Then, there exists a constant $C\in (0,\infty)$ such that 
\begin{equation}
\mathbb{E} [\dist^2(0,X_t)] \leq Ct \mu^{\frac{\tilde{\nu} \tilde{\alpha}_1}{1+2\tilde{\nu}}}\left[\log(1/\mu)\right]^{1-\frac{\tilde{\nu} \tilde{\alpha}_1}{1+2\tilde{\nu}}},~\forall t \geq \mu^{\frac{-2\tilde{\nu}}{1+2\tilde{\nu}}}\left[\log(1/\mu)\right]^{\frac{-1}{1+2\tilde{\nu}}}, ~ \mu\in(0,1/\ee].
\end{equation}
\end{proposition}

It is believed that for $2\leq d <6$, we have $\tilde{\alpha}_1=\alpha_1\in(0,2)$ and $\tilde{\nu}=\nu\geq 1/2$. But, the fact that $\tilde{\alpha}_1\in(0,2)$ has only been proved rigorously when $d=2$, as shown in Lemmas \ref{lem:onearm2d_triangle} and \ref{lem:onearm2d_Z} below. Hence, currently, we can only apply the above proposition to the $d=2$ case.

To prove \Cref{prop:ublowd}, we need some general results on the mean squared displacement, which is valid for all $p \in [0,1]$. We first give an estimate on the distribution function of the displacement using the forward/backward martingale decomposition of reversible Markov processes, which can be found in Kesten~\cite[Proposition~3.3]{Kes86} and \cite[Theorem 2.3]{NPSS2006}. 
	
	\begin{proposition}\label{prop:tailGeneral}
		Let $G=\mathcal{T}$ or $\mathbb{Z}^d$ with $d\geq 1$. There exists a constant $C_{\ref{prop:tailGeneral}}(d) \in [1,\infty)$ such that 
		\begin{equation}\label{eq:tailGeneral}
			\mathbb{P} (\dist(0, X_t) \geq L) \leq \left\{\begin{array}{cc}
				C_{\ref{prop:tailGeneral}}\exp\left(- \frac{L^2}{C_{\ref{prop:tailGeneral}} t}\right), & \qquad \text{if } L \leq 2t, \\
				C_{\ref{prop:tailGeneral}}\exp\left(- \frac{2 L}{C_{\ref{prop:tailGeneral}}}\right), & \qquad \text{if } L > 2t. 
			\end{array}\right. 
		\end{equation}
	\end{proposition}
	\begin{proof}
		We treat at first the case of $\mathbb{Z}^d$. In the following proof, we denote by 
		\begin{align}\label{eq.defJump}
			N_t := \text{ number of attempted jumps of the random walk in } [0,t].
		\end{align}
		Since the random walk attempts to jump with rate $1$, the random variable $N_t$ has a Poisson distribution of parameter $t$. Especially, it satisfies the Chernoff bound (see \cite[Exercise 2.2.23]{dembo-zeitouni})
		\begin{align}\label{eq.NtChernoff}
			\mathbb{P}(N_t \geq n) \leq \exp(n - t - n \log(n/t) ), \qquad \forall n \geq t.
		\end{align}
		Then we can use $N_t$ to give an upper bound
		\begin{equation}\label{eq.expTail}
				 \mathbb{P} (\dist(0, X_t) \geq L) \leq \mathbb{P}(N_t \geq L) \leq  \exp(L - t - L \log(L/t) ) \leq \exp(-C L),~\forall L \geq 2t.
		\end{equation}
		
		In the remaining part, we focus on the case $L \leq 2t$. We have
		\begin{equation}\label{eq.smallDisplacement}
			\begin{split}
				\mathbb{P} (\dist(0, X_t) \geq L) &\leq \mathbb{P} (\dist(0, X_t) \geq L, N_t < 2t) + \mathbb{P}(N_t \geq 2t) \\
				&\leq  \mathbb{P} (\dist(0, X_t) \geq L, N_t < 2t) + \exp(-C t),
			\end{split}
		\end{equation}
    where we have applied \eqref{eq.expTail} to the term $\mathbb{P}(N_t \geq 2t)$. On the event $N_t < 2t$, we can make a natural coupling between our process and another one defined on a large torus. We denote by $\lceil t \rceil := \lfloor t \rfloor + 1$ and let $\T^d_{6\lceil t \rceil }$ be the discrete torus with side length $6\lceil t \rceil$ (i.e., $B_{3\lceil t \rceil}$ with opposite vertices of $\partial B_{3\lceil t \rceil}$ identified), and denote by $\overline{\mathbb{P}}_0$ the probability that the random walk starts at $0$, with $\mathrm{Ber}(p)^{E(\T^d_{6\lceil t \rceil})}$ as the initial distribution of the dynamical percolation. Using this natural coupling, we get 
		\begin{align}\label{eq.pass1}
	\mathbb{P} (\dist(0, X_t) \geq L, N_t < 2t) &\leq \overline{\mathbb{P}}_0 (\dist(0, X_t) \geq L).
		\end{align}
		It suffices to estimate the probability on the RHS, and it is more convenient to randomize the starting point. Let $\overline{\mathbb{P}}$ be the probability that the random walk $X_0$ starts from a uniformly chosen point on $\T^d_{6\lceil t \rceil}$, with $\mathrm{Ber}(p)^{E(\T^d_{6\lceil t \rceil})}$ as the initial distribution of dynamical percolation. The process $(X_t, \eta_t)_{t \geq 0}$ is stationary under $\overline{\mathbb{P}}$, which implies
		\begin{equation}\label{eq.pass2}
              \overline{\mathbb{P}}_0 (\dist(0, X_t) \geq L) = \overline{\mathbb{P}}(\dist(X_0, X_t) \geq L) \leq \sum_{i=1}^d \overline{\mathbb{P}}\left(\vert h_i(X_t) - h_i(X_0)\vert \geq L/d \right),
		\end{equation}
        where $h_i(X_t) \equiv h_i(X_t, \eta_t)$ is the projection onto the $i$-th coordinate of $X_t$.
		
		We then use the forward/backward martingale decomposition for stationary reversible Markov process; see \cite[Lemma~13.15]{LP16} for one example. Denoting by $\mathscr{L}$ the generator of $(X_s, \eta_s)_{s \geq 0}$, we have the martingale decomposition of $h_i(X_t, \eta_t)$
		\begin{align}\label{eq.defmt}
			M_t := h_i(X_t,\eta_t) - h_i(X_0, \eta_0) - \int_0^t (\mathscr{L} h_i)(X_s, \eta_s) \d s.
		\end{align}
		We now define the reverse process $(\widetilde{X}_s, \widetilde{\eta}_s)_{0 \leq s \leq t} := (X_{t-s}, \eta_{t-s})_{0 \leq s \leq t}$. Since the process $(X_s, \eta_s)_{s \geq 0}$ is stationary and reversible under $\overline{\mathbb{P}}$, the reverse process $(\widetilde{X}_s, \widetilde{\eta}_s)_{0 \leq s \leq t}$ has the same law as $(X_s, \eta_s)_{0 \leq s \leq t}$, and
		\begin{align}\label{eq.defmtR}
			\widetilde{M}_s := h_i(\widetilde{X}_s, \widetilde{\eta}_s) - h_i(\widetilde{X}_0, \widetilde{\eta}_0) - \int_0^s (\mathscr{L} h_i)(\widetilde{X}_r, \widetilde{\eta}_r) \d r.
		\end{align}
		also defines a martingale $(\widetilde{M}_s)_{0 \leq s \leq t}$ with respect to the natural filtration of $(\widetilde{X}_s, \widetilde{\eta}_s)_{0 \leq s \leq t}$. Inserting its definition, the backward martingale can also be expressed as  
		\begin{align*}
			\widetilde{M}_t :=  h_i(X_0,\eta_0) - h_i(X_t, \eta_t) - \int_0^t (\mathscr{L} h_i)(X_s, \eta_s) \d s.
		\end{align*}
		Subtracting this expression from \eqref{eq.defmt} to cancel the drift term, we obtain that
		\begin{align}\label{eq.mtdifference}
			2 \big(h_i(X_t) - h_i(X_0)\big) =  M_t - \widetilde{M}_t.
		\end{align}
		Then, the tail estimate in \eqref{eq.pass2} is transformed to that of martingale
		\begin{align}\label{eq.pass3}
			\overline{\mathbb{P}}\left(\vert h_i(X_t) - h_i(X_0)\vert \geq L/d \right) \leq 	\overline{\mathbb{P}}\left(\vert M_t\vert \geq L/d \right) + \overline{\mathbb{P}}\left(\vert \widetilde{M}_t \vert \geq L/d \right).
		\end{align}
		It suffices to apply an Azuma-Hoeffding type estimate. For the martingale associated with the jump process of bounded quadratic variation, we refer to the version in \cite[Lemma~A.2]{GGM19} and conclude 
		\begin{align}\label{eq.mtConcentration}
			\overline{\mathbb{P}}\left(\vert M_t\vert \geq L/d \right) \leq 2\exp\left(2d (e^{\lambda} - 1 - \lambda) t  - \lambda L/d  \right).
		\end{align}
		Then we optimize the value by choosing $\lambda = \lambda^* = \log\left(1 + \frac{L}{2 d^{2}t} \right)$. When $L \leq 2t$, we have $2d (e^{\lambda^*} - 1 - \lambda^*) t  - \lambda^* L/d\leq (\log 2-1)L^2/(2d^3t)$, and thus
		\begin{align}\label{eq.mtConcentrationBound}
			\overline{\mathbb{P}}\left(\vert M_t\vert \geq L/d \right) \leq 2\exp(-L^2/(C_1 t)),~\forall L\leq 2t.
		\end{align}
		A similar estimate also works for $ \widetilde{M}_t$. We combine \eqref{eq.mtConcentrationBound}, \eqref{eq.pass3}, \eqref{eq.pass2}, \eqref{eq.pass1}, \eqref{eq.smallDisplacement} and conclude
		\begin{align*}
			\mathbb{P} (\dist(0, X_t) \geq L) \leq 4d\exp(-L^2/(C_1 t)) + \exp(-C t), \forall L \leq 2t.
		\end{align*}
		The first term dominates in the regime $L \leq 2t$. Combining this with \eqref{eq.expTail} and changing the value of $C$, we obtain the desired result. 
		
		For the case $\mathcal{T}$, we can embed it in $\mathbb{Z}^2$ by adding one diagonal in every unit square, and then the same proof applies. 
	\end{proof}
	
	\begin{remark}
		We can also deduce the Carne-Varopoulos bound for $\mathbb{P}(X_t = y)$ using a similar proof as above. See the argument by R\'emi Peyre in \cite{Peyre08},  \cite[Theorem~1.3]{GGM19} for a similar result of random walk in Kawasaki dynamics, and also Proposition 79 of \cite{Buc11}. However, to deduce a tail probability for $\dist(0,X_t)$ from the Carne-Varopoulos bound will bring another volume factor. As we are more interested in $\dist(0,X_t)$, we use the stationarity of process to give a direct proof.    
	\end{remark}

	The estimate in Proposition~\ref{prop:tailGeneral} yields an upper bound of the conditional expectation of the squared displacement given some event.
	\begin{proposition}\label{prop:tlogt}
		Let $G=\mathcal{T}$ or $\mathbb{Z}^d$ with $d\geq 1$. There exists  a constant $C_1=C_{1}(d) \in (1,\infty)$ such that for any event $A$ determined by the dynamical environment with $\mathbb{P}(A) = \epsilon \in (0,1]$, we have
		\begin{equation}\label{eq:tlogt}
			\mathbb{E} [\dist^2(0,X_t) | A]\leq C_{1} \left(1 + \max\left\{t \log\left( \frac{C_{1}}{\epsilon}\right), \log^2\left(\frac{C_{1}}{\epsilon}\right)\right\}\right). 
		\end{equation}
	\end{proposition}
	\begin{proof}
		Throughout the proof, we denote by $\Phi$ the upper bound in \eqref{eq:tailGeneral}, 
		\begin{align}\label{eq.defPhi}
			\Phi(r) := \left\{\begin{array}{cc}
				C_{\ref{prop:tailGeneral}}\exp\left(- \frac{r^2}{C_{\ref{prop:tailGeneral}} t}\right), & \qquad \text{if } r \leq 2t, \\
				C_{\ref{prop:tailGeneral}}\exp\left(- \frac{2r}{C_{\ref{prop:tailGeneral}}}\right), & \qquad \text{if } r > 2t. 
			\end{array}\right.
		\end{align}
		which is decreasing and continuous.  We calculate the expectation using the following expression
		\begin{equation}\label{eq.squareDecom}
			\begin{split}
				\mathbb{E} [\dist^2(0,X_t) \mathbf{1}_{A}] &= \int_0^\infty 2 r \mathbb{P}(\dist(0,X_t) \mathbf{1}_{A} > r) \, \d r \leq \int_0^\infty 2 r \min\{\Phi(r), \epsilon\} \, \d r \\
				&=  \int_0^{r_0} 2 r \epsilon \, \d r + \int_{r_0}^\infty 2 r \Phi(r) \, \d r =  {r_0}^2 \epsilon + \int_{r_0}^\infty 2 r \Phi(r) \, \d r,
			\end{split}
		\end{equation}
		where $r_0$ is the threshold satisfying 
		\begin{align}\label{eq.defr0}
			r_0 := \sup\{r \geq 0 : \Phi(r) > \epsilon\},
		\end{align}
		whose value is $\Phi^{-1}(\epsilon)$ but depends on the regime in \eqref{eq.defPhi}. 
		
		We calculate the last line in \eqref{eq.squareDecom}, and just denote by $C \equiv C_{\ref{prop:tailGeneral}}$ for convenience. 
		
		\smallskip
		
		\textit{Case I: $\epsilon  \geq \Phi(2t) = C\exp\left(- \frac{4t}{C}\right) $.} For this case, we have  $r_0 = \sqrt{C t \log\left(\frac{C}{\epsilon}\right)} \in (0,2t]$ and
		\begin{equation}\label{eq.squareCase1Decom}
			\begin{split}
				\int_{r_0}^\infty 2 r \Phi(r) \, \d r &= \int_{r_0}^{2t} 2 r \Phi(r) \, \d r + \int_{2t}^\infty 2 r \Phi(r) \, \d r, \\
				\int_{r_0}^{2t} 2 r \Phi(r) \, \d r &\leq \int_{r_0}^\infty 2 r C\exp\left(- \frac{r^2}{C t}\right) \, \d r = C^2 t\exp\left(- \frac{r_0^2}{C t}\right) = C t \epsilon, \\
				\int_{2t}^\infty 2 r \Phi(r) \, \d r &= \int_{2t}^\infty 2 r C\exp\left(- \frac{2r}{C}\right) \, \d r = \exp\left(- \frac{4t}{C}\right)(2 C^2 t + C^3/2) \leq (2 Ct +  C^2)\epsilon.
			\end{split}
		\end{equation}
		In the second line, we use the fact $\Phi(r_0) = C\exp\left(- \frac{r_0^2}{C t}\right) =\epsilon$; in the third line, we use the condition $\Phi(2t) \leq \epsilon$. We put the expression of $r_0$ and \eqref{eq.squareCase1Decom} back to \eqref{eq.squareDecom} to get
		\begin{align}\label{eq.squareCase1}
			\mathbb{E} [\dist^2(0,X_t) \mathbf{1}_{A}] \leq \left(3 Ct +  C^2 + C t \log\left(\frac{C}{\epsilon}\right)\right) \epsilon.
		\end{align}
		
		\smallskip
		
		\textit{Case II: $\epsilon < \Phi(2t) = C\exp\left(- \frac{4t}{C}\right) $.} For this case, we have $r_0 = \frac{C}{2} \log\left(\frac{C}{\epsilon}\right) \in (2t, \infty)$. The following calculation is similar to the third line of \eqref{eq.squareCase1Decom}  
		\begin{equation*}
			\int_{r_0}^\infty 2 r \Phi(r) \, \d r  =\exp\left(- \frac{2 r_0}{C}\right)( C^2 r_0 +  C^3/2) = (Cr_0 + C^2/2)\epsilon.	
		\end{equation*}
		%Here we also use the condition $\Phi(r_0) = C\exp\left(- \frac{2r_0}{C}\right)= \epsilon$ to deduce the last equality. 
       Plugging this and the expression of $r_0$  into \eqref{eq.squareDecom}, we obtain
		\begin{equation}\label{eq.squareCase2}
			\begin{split}
				\mathbb{E} [\dist^2(0,X_t) \mathbf{1}_{A}] &\leq ({r_0}^2 +  Cr_0 + C^2/2) \epsilon =  C^2 \left(\frac{1}{4}\log^2\left(\frac{C}{\epsilon}\right) + \frac{1}{2}\log\left(\frac{C}{\epsilon}\right) + \frac{1}{2} \right) \epsilon.
			\end{split}
		\end{equation}
		Combining \eqref{eq.squareCase1} and \eqref{eq.squareCase2}, we get \eqref{eq:tlogt}.
	\end{proof}
	
	\begin{remark}
		In our applications, the event $A$ has probability $\mathbb{P}(A) \asymp t^{-\alpha}$ with $\alpha > 0$, so we obtain $t \log t$ as an upper bound. 
		%The extra $\log t$ factor actually can not be removed, as shown in a static graph by Barlow and Perkins~\cite{BP89}.
	\end{remark}

 \begin{proof}[Proof of Proposition \ref{prop:ublowd}]
 We denote by $\HH$ the subgraph of $\mathcal{T}$ (resp.,~$\mathbb{Z}^d$) composed of all the sites (resp.,~bonds) open at least once during $[0,t]$. Then, $\HH$ is also a static percolation cluster on $\mathcal{T}$ (resp.,~$\mathbb{Z}^d$), where each site (resp.,~bond) is open with probability
 	\begin{equation}\label{eq:ppert}
 		p := p_c + (1-p_c)(1-  \ee^{-\mu t p_c}) \leq p_c(1+\mu t).
 	\end{equation}
 	Let $\{0 \stackrel{\HH}{\longleftrightarrow} \partial B_r\}$ be the event that the origin is connected to some vertex in $\partial B_r$ via sites (resp.,~bonds) in $\HH$. 
 	
Let $K$ be a threshold that will be chosen later. We have
 	\begin{align}
 		\e[\dist^2(0, X_t)]
 		&\leq \sum_{k=1}^K 	\e\left[\dist^2(0, X_t)\mathbf{1}\{0 \stackrel{\HH}{\longleftrightarrow} \partial B_{2^{k-1}}, 0 \stackrel{\HH}{\centernot\longleftrightarrow} \partial B_{2^{k}}\}\right] + \e\left[\dist^2(0, X_t) \mathbf{1} \{0 \stackrel{\HH}{\longleftrightarrow} \partial B_{2^{K}}\} \right]\notag\\
 		&\leq C\sum_{k=1}^K 4^k \p (0 \stackrel{\HH}{\longleftrightarrow} \partial B_{2^{k-1}}) +\e\left[\dist^2(0, X_t) \mathbf{1} \{0 \stackrel{\HH}{\longleftrightarrow} \partial B_{2^{K}}\} \right],\label{eq:dist^2ub}
 	\end{align}
 	where we used the trivial bound $\dist^2(0, X_t)\leq C4^k$ under the event $0 \stackrel{\HH}{\centernot\longleftrightarrow} \partial B_{2^{k}}$ in the last inequality. If we choose $K\in \mathbb{N}\cup\{0\}$ (so we implicitly assume that $p_c\mu t\leq 1$) such that \[2^{-(K+1)/\tilde{\nu}}< p_c\mu t\leq 2^{-K/\tilde{\nu}},\] 
  then \eqref{eq:ppert} and Assumption \ref{a:1arm} imply that
 	\begin{equation*}
 	 C_0 2^{-k\tilde{\alpha}_0} \leq \p (0 \stackrel{\HH}{\longleftrightarrow} \partial B_{2^{k}}) \leq C_1 2^{-k\tilde{\alpha}_1}, ~\forall 0\leq k\leq K.
 	\end{equation*}
Recall that $p_c=p_c(G)\leq 1/2$ for $G=\mathcal{T}$ or $\mathbb{Z}^d$ with $d\geq 2$.
If we set $T:=\mu^{\frac{-2\tilde{\nu}}{1+2\tilde{\nu}}}[\log(1/\mu)]^{\frac{-1}{1+2\tilde{\nu}}}$, then we have $\mu T\leq 1$ and $T\geq \sqrt{\ee}$ since $e\mu\leq \log(1/\mu)\leq (\ee\mu)^{-2\tilde{\nu}}$ for each $0<\mu\leq 1/\ee$ and $2\tilde{\nu}\geq 1$; we also have 
\begin{equation}
    |\log \p (0 \stackrel{\HH}{\longleftrightarrow} \partial B_{2^{K}})| \asymp \log t,~\forall T\leq t \leq 2T.
\end{equation}
So for such $t\in[T,2T]$, the dominant term for the upper bound of the expectation in \eqref{eq:dist^2ub} from Proposition \ref{prop:tlogt} is the single $\log$ term. Combining this with \eqref{eq:dist^2ub}, we get
\begin{align}\label{eq:dist2overt}
    \e[\dist^2(0, X_t)]&\leq C 2^{(2-\tilde{\alpha}_1)K}+C_2(t\log t) 2^{-K\tilde{\alpha}_1}\leq  C(\mu t)^{-(2-\tilde{\alpha}_1)\tilde{\nu}}+C_3( t \log t) (\mu t)^{\tilde{\nu} \tilde{\alpha}_1}\nonumber\\
    &\leq  C_4 t\mu^{\frac{\tilde{\nu} \tilde{\alpha}_1}{1+2\tilde{\nu}}}[\log(1/\mu)]^{1-\frac{\tilde{\nu} \tilde{\alpha}_1}{1+2\tilde{\nu}}},\ \ \forall T\leq t \leq 2T.
\end{align}

Now, consider the random walk defined on the discrete torus $\mathbb{T}_n^d$ with side length $2n$ and $d \geq 2$. If we start the full process $(X_t, \eta_t)$ in stationary with initial law $u\otimes \pi_p$ (where $u$ is the uniform distribution on $\mathbb{T}_n^d$), the following estimate holds:
  \begin{equation}
      \e_{\mathbb{T}_n^d}[\dist^2(X_0, X_{ks})]\leq  C_5 k \e_{\mathbb{T}_n^d}[\dist^2(X_0, X_s)], \ \ \forall k\in \mathbb{N}, ~ s\geq0,
  \end{equation}
  where $C_5\in(0,\infty)$ is a constant independent of $n, k$ and $s$. Such an estimate can be obtained by using the Markov type $2$ property, which was introduced and studied by Ball \cite{Ball1992} and then further developed by Naor, Peres, Schramm, and Sheffield \cite{NPSS2006}; see also Lemma 4.2 of \cite{peres-stauffer-steif}. By symmetry, for the full system starting with initial law $\delta_0\otimes \pi_p$, we also have
    \begin{equation}
      \e_{\mathbb{T}_n^d}[\dist^2(0, X_{ks})]\leq  C_5 k \e_{\mathbb{T}_n^d}[\dist^2(0, X_s)],\ \ \forall k\in \mathbb{N}, ~ s\geq0.
  \end{equation}
  Setting $n\rightarrow \infty$ in the last inequality, we obtain that on $\mathcal{T}$ (or $\mathbb{Z}^d$),
  \begin{equation}\label{eq:dist^2tri}
      \e[\dist^2(0, X_{ks})]\leq  C_5 k \e[\dist^2(0, X_s)], \ \ \forall k\in \mathbb{N}, ~  s\geq0.
  \end{equation}

  Finally, combining \eqref{eq:dist2overt} and \eqref{eq:dist^2tri}, we get that
  \[\e[\dist^2(0, X_{t})]\leq C_6 t \mu^{\frac{\tilde{\nu} \tilde{\alpha}_1}{1+2\tilde{\nu}}}[\log(1/\mu)]^{1-\frac{\tilde{\nu} \tilde{\alpha}_1}{1+2\tilde{\nu}}}, \ \ \forall t \geq T, \]
which completes the proof of the proposition.
 \end{proof}
 
\subsection{Proof of Theorem \ref{t:ubtri}} 
We give a stability result concerning the one-arm probability for near-critical site percolation on $\mathcal{T}$.
\begin{lemma}\label{lem:onearm2d_triangle}
 	For site percolation on $\mathcal{T}$, for each $\epsilon\in(0,5/48)$, there exist constants $C_0=C_0(\epsilon), C_1=C_1(\epsilon)\in(0,\infty)$ such that
 	\begin{equation}
 		C_0 r^{-\frac{5}{48}-\epsilon} \leq \mathbb{P}_p(0\longleftrightarrow \partial B_r)\leq C_1 r^{-\frac{5}{48}+\epsilon}, \ \ \forall r\geq 1, ~ p\in[1/2,1/2+r^{-(3+\epsilon)/4}].
 	\end{equation}
 \end{lemma}
 \begin{proof}
 	Theorem 22 of \cite{Nol08} gives that
 	\[\lim_{p\rightarrow 1/2}\frac{\log L(p)}{\log |p-1/2|}=-\frac{4}{3},\]
 	where $L(p)$ represents the characteristic length. Consequently, for any $\epsilon>0$, there exists $\delta=\delta(\epsilon)>0$ such that
 	\begin{equation}\label{eq:cl}
 		L(p)\geq (p-1/2)^{-4/(3+\epsilon)},~\forall p\in(1/2,1/2+\delta).
 	\end{equation}
On the other hand, the main result in \cite{LSW02} gives that
 	\[\lim_{r\rightarrow\infty}\frac{\mathbb{P}_{1/2}(0\longleftrightarrow \partial B_r)}{\log r}=-\frac{5}{48},\]
 	which implies that for each $\epsilon\in(0,5/48)$, there exists $N=N(\epsilon)>0$ such that
 	\begin{equation}\label{eq:onearm2d}
 		r^{-\frac{5}{48}-\epsilon} \leq \mathbb{P}_{1/2}(0\longleftrightarrow \partial B_r) \leq r^{-\frac{5}{48}+\epsilon},\ \ \forall r\geq N.
 	\end{equation}
 	The LHS gives the lower bound in the lemma. Theorem 1 of \cite{Kes87} (see also Theorem 27 of \cite{Nol08})  implies the existence of a constant  $C_1\in(0,\infty)$ (independent of $p$) such that
 	\begin{equation}\label{eqn::criticalwindow}
 		\mathbb{P}_{p}(0\longleftrightarrow \partial B_n) \leq C_1\mathbb{P}_{1/2}(0\longleftrightarrow \partial B_n),\ \ \forall 0\leq n \leq L(p).
 	\end{equation}
 	This combined with \eqref{eq:cl} gives that if $r>\delta^{-\frac{4}{3+\epsilon}}$, then
 	\begin{equation}
 		\mathbb{P}_{p}(0\longleftrightarrow \partial B_n) \leq C_1\mathbb{P}_{1/2}(0\longleftrightarrow \partial B_n),~\forall p\in[1/2, 1/2+r^{-\frac{3+\epsilon}{4}}], ~ 0\leq n \leq r.
 	\end{equation}
 	Together with \eqref{eq:onearm2d}, this estimate implies that
 	\begin{equation}
 		\mathbb{P}_{p}(0\longleftrightarrow \partial B_r) \leq C_1 r^{-\frac{5}{48}+\epsilon}, ~\forall p\in[1/2, 1/2+r^{-\frac{3+\epsilon}{4}}], ~ r> \max\{N,\delta^{-\frac{4}{3+\epsilon}}\}.
 	\end{equation}
  Then, the proof of the lemma is completed by potentially increasing $C_1$ to a larger constant $C_1$ in order to accommodate the case $1\leq r \leq \max\{N,\delta^{-\frac{4}{3+\epsilon}}\}$.
 	%This completes the proof of the lemma by possibly changing $C_0'$ to a larger constant $C_0$ in order to accommodate $1\leq r \leq \max\{N,\delta^{-\frac{4}{3+\epsilon}}\}$.
 \end{proof}
  \begin{proof}[Proof of Theorem \ref{t:ubtri}]
Combining Proposition \ref{prop:ublowd} and Lemma \ref{lem:onearm2d_triangle}, we immediately obtain that for each $\epsilon\in(0,5/132)$, there exists a constant $C=C(\epsilon)\in (0,\infty)$ such that for every $\mu\in(0,1/\ee]$
  \begin{equation}\label{eq:triub}
		\mathbb{E} [\dist^2(0,X_t)] \leq C t \mu^{\frac{5-48\epsilon}{132+12\epsilon}}\left[\log(1/\mu)\right]^{1-\frac{5-48\epsilon}{132+12\epsilon}},~ \quad \forall t \geq \mu^{\frac{-8}{11+\epsilon}}[\log(1/\mu)]^{-\frac{3+\epsilon}{11+\epsilon}}.
	\end{equation}
This gives~\eqref{eqn::conclusionT} as desired.  
 \end{proof}

For the bond percolation on $\mathbb{Z}^2$, it is known that the one-arm exponent satisfies $\alpha_1\leq 1/6$ (see Section 6.4 of Duminil-Copin, Manolescu, and Tassion \cite{DCMT2021}). 
However, we have not been able to find any lower bound for $\alpha_1$ in the literature, except that in Theorem~1.12 of \cite{DM23} by Dewan and Muirhead, a lower bound is provided under the assumption of the existence of other critical exponents. The following lemma gives a rough lower bound on $\alpha_1$.  When combined with Proposition~\ref{prop:ublowd}, this further yields an upper bound for the mean squared displacement of the random walk on dynamical percolation on $\mathbb{Z}^2$ at criticality. The precise statement can be found in Corollary \ref{cor:Z^2} below.

%this gives a upper bound for the mean squared displacement for the random walk on the dynamical percolation on $\mathbb{Z}^2$ at criticality, see Corollary \ref{cor:Z^2} below.

\begin{lemma}\label{lem:onearm2d_Z}  
For bond percolation on $\mathbb{Z}^2$, there exist constants $\tilde{\alpha}_0, C_0, C_1\in (0,\infty)$ and $\tilde{\alpha}_1=-\frac{1}{2}\log(1-a)$ with $a:=2^{-24}(1-\sqrt{3}/2)^{48}$ such that
 	\begin{equation}\label{eqn::onearmZ2}
 		\frac{C_0}{r^{\tilde{\alpha}_0}} \leq \mathbb{P}_p(0\longleftrightarrow \partial B_r)\leq \frac{C_1}{r^{\tilde{\alpha}_1}}, \ \ \forall r\geq 1, ~ p\in[1/2,1/2+r^{-2}].
 	\end{equation}
  \end{lemma}  
  \begin{proof}
The lower bound follows from Theorem 1.1 of \cite{BE22} or RSW type argument \cite{Ruo78,SW78}. So we focus on the upper bound next. Let $A_{l,3l}:=B_{3l}\setminus B_l$ for $l\in\mathbb{N}$. Using the inequality (11.72) from Grimmett \cite{Gri99}, we get that
\begin{equation}\label{eq:opencircuit}
\mathbb{P}_{1/2}(\exists \text{ open circuit in the annulus }A_{l,3l} \text{ surrounding the origin})\geq a=2^{-24}(1-\sqrt{3}/2)^{48}.
\end{equation}
The number of disjoint annuli of the form $(1/2,1/2)+A_{l,3l}, l\in\mathbb{N}$, in the dual lattice $\mathbb{Z}^{2*}:=(1/2,1/2)+\mathbb{Z}^2$, that are contained in $B_r$, is at least $(\log r)/2$ for each $r\geq N_0$, where $N_0\in\mathbb{N}$ is a fixed integer. A dual edge $e^*\in E(\mathbb{Z}^{2*})$ is declared to be open if and only if it crosses a closed edge in $E(\mathbb{Z}^2)$. Similar to \eqref{eq:opencircuit}, we have 
\[\mathbb{P}_{1/2}(\exists \text{ dual open circuit in } (1/2,1/2)+A_{l,3l})\geq a, \ \ \forall l\in \mathbb{N}.\]
Now, we obtain that
\begin{align}\label{eq:onearm2dZc}
\mathbb{P}_{1/2}(0\longleftrightarrow \partial B_r)&\leq \mathbb{P}_{1/2}(\text{none of those } (\log r)/2 \text{ annuli has a dual open circuit})\nonumber\\
&\leq (1-a)^{\frac{1}{2}\log r}=r^{\frac{1}{2}\log(1-a)},\ \ \forall r\geq N_0.
\end{align}
Note that the critical percolation in $B_r$ with parameter $1/2$ can be obtained by first performing percolation with parameter $p>1/2$ and then independently deleting each open edge with probability $1-1/(2p)$. Thus, we have that
      \[\mathbb{P}_{1/2}(0\longleftrightarrow \partial B_r)\geq \left(\frac{1}{2(1/2+r^{-2})}\right)^{\#\text{ of edges in }B_r} \mathbb{P}_{1/2+r^{-2}}(0\longleftrightarrow \partial B_r), \ \ \forall r\geq 1.\]
      This combined with \eqref{eq:onearm2dZc} completes the proof of the lemma.
  \end{proof}

We are ready to prove the following corollary, which implies \Cref{prop:Z^2}.
\begin{corollary}\label{cor:Z^2}
Consider the random walk on the dynamical percolation on $\mathbb{Z}^2$ at criticality $p=1/2$, starting from $X_0=0$ and the initial law $\pi_{1/2}:=\mathrm{Ber}(1/2)^{E(\mathbb Z^2)}$ for the dynamical percolation. Then, there exists a constant $C\in (0,\infty)$ (independent of $\mu$) such that
    \begin{equation}\label{eq:Z^2ub}
		\mathbb{E} [\dist^2(0,X_t)] \leq C t \mu^{\frac{\tilde{\alpha}_1}{4}}\left[\log(1/\mu)\right]^{1-\frac{ \tilde{\alpha}_1}{4}},\ \ \forall t \geq \mu^{-\frac{1}{2}}\left[\log(1/\mu)\right]^{-\frac{1}{2}}, ~ \mu\in(0,1/\ee],
	\end{equation}
 where $\tilde{\alpha}_1=-\frac{1}{2}\log(1-a)$ with $a=2^{-24}(1-\sqrt{3}/2)^{48}$.
\end{corollary}

\begin{proof}
    This follows directly from Lemma~\ref{lem:onearm2d_Z} and Proposition~\ref{prop:ublowd}. 
\end{proof}

\subsection{Proof of Theorem \ref{t:ubhigh}}
We first state a lemma similar to Lemmas \ref{lem:onearm2d_triangle} and \ref{lem:onearm2d_Z} that provides the stability of the one-arm probability for high-dimensional percolation. 
 \begin{lemma}	\label{lem:onearmhd}
Fix $d\geq 11$.  For bond percolation on $\mathbb{Z}^d$, there exist constants $ C_0 \le C_1\in (0,\infty)$ such that 
\begin{equation}\label{eq:stabonearm_high}
 		\frac{C_0}{r^2} \le \p_p(0\longleftrightarrow \partial B_r) \le \frac{C_1}{r^2},  \ \ \forall r\geq 1, ~ p\in[p_c(\mathbb{Z}^d), p_c(\mathbb{Z}^d)+r^{-2}].
 	\end{equation}  
 \end{lemma}
 
\begin{remark}
    In an early version of the paper, we included a proof of Lemma \ref{lem:onearmhd} due to Tom Hutchcroft (private communication), which used an inequality involving decision trees from \cite{Hut22}. The alternative proof we present below was indicated to us by Fedor Nazarov in a lecture.  We are grateful to both of them.
\end{remark}
 \begin{proof}
 Theorem~1 of \cite{KN11} implies the existence of constants $C_0,C_1\in(0,\infty)$ such that
 \begin{equation}\label{eq:criticaldecay}
 \frac{C_0}{r^2} \le \p_{p_c}(0\longleftrightarrow \partial B_r) \le \frac{C_1}{r^2}, \ \ \forall r\in\mathbb{N}\, .\end{equation}
 	This gives the lower bound in \eqref{eq:stabonearm_high}. For the upper bound, we first recall a celebrated result on intrinsic one-arm exponent due to Kozma and Nachmias \cite{KN09}. Let $d_p(x,y)$ be the length of the shortest open path between $x$ and $y$ in bond percolation on $\mathbb{Z}^d$ with parameter $p$, where $d_p(x,y):=\infty$ if no such open path exists.
  %; $d_p$ is usually called the {\it intrinsic metric} (i.e., the graph distance on $p$-bond percolation on $\mathbb{Z}^d$). 
  Theorem 1.2 of \cite{KN09} and Corollary 1.3 of \cite{FvdH17} imply that there exists $C\in(0,\infty)$ such that
  \begin{equation}\label{eq:intrinsic}
     \p(\exists x\in \mathbb{Z}^d \text{ such that } d_{p_c}(0,x)=r)\leq \frac{C}{r},~\forall r\in\mathbb{N}. 
  \end{equation}
  We may write
  \begin{equation}\label{eq:onearm1}
      \p_p(0\longleftrightarrow \partial B_r) \leq \p(\exists x\in\partial B_r :  \; d_{p}(0,x)\leq r^2)+\p(\exists x\in\partial B_r : \; r^2<d_{p}(0,x)< \infty).
  \end{equation}

 Percolation    with parameter $p_c$ can be obtained by first performing percolation with parameter $p>p_c$, and then independently closing each open edge with probability $1-p_c/p$. Therefore,
  \begin{equation}\label{eq:criticalpert}
      \p(\exists x\in\partial B_r \text{ such that } d_{p_c}(0,x)\leq r^2)\geq \left(\frac{p_c}{p}\right)^{r^2}\p(\exists x\in\partial B_r \text{ such that } d_{p}(0,x)\leq r^2),
  \end{equation}
  where the factor $(p_c/p)^{r^2}$ is a lower bound of the probability that each of the $\leq r^2$ edges in a path from $0$ to $x$ (which is open after the first step), is still open after the second step. Thus,
\begin{align}\label{eq:onearm11}
      \p(\exists x\in\partial B_r  : \; d_{p}(0,x)\leq r^2)
      &
      \leq \left(\frac{p}{p_c}\right)^{r^2}\p(\exists x\in\partial B_r : \; d_{p_c}(0,x)\leq r^2)
      \nonumber\\
      &
\leq\left(\frac{p_c+r^{-2}}{p_c}\right)^{r^2} \p_{p_c}(0\longleftrightarrow \partial B_r)\leq \frac{C_2}{r^2},
  \end{align}
  where we have used \eqref{eq:criticaldecay} in the last inequality.

  For the second term on the RHS of \eqref{eq:onearm1}, we have
  \begin{align}\label{eq:onearm12}
      \p(\exists x\in\partial B_r \text{ such that } r^2<d_{p}(0,x)< \infty) &\leq \p(\exists y\in B_r \text{ such that } d_{p}(0,y)= r^2)\nonumber\\
      &\leq \left(\frac{p}{p_c}\right)^{r^2}  \p(\exists y\in B_r \text{ such that } d_{p_c}(0,y)= r^2)\nonumber\\
      &\leq \left(\frac{p_c+r^{-2}}{p_c}\right)^{r^2} \frac{C}{r^2},
  \end{align}
  where we have used a similar argument as in \eqref{eq:criticalpert} in the second inequality and \eqref{eq:intrinsic} in the last inequality. Plugging \eqref{eq:onearm12} and \eqref{eq:onearm11} into \eqref{eq:onearm1} completes the proof of the upper bound in~\eqref{eq:stabonearm_high}.
  \end{proof}

 \begin{proof}[Proof of Theorem \ref{t:ubhigh}]
 	The proof strategy is similar to that of Proposition \ref{prop:ublowd}, so we only highlight the main differences. We denote by $\HH$ the subgraph composed of all the bonds open at least once during $[0,t]$. For some $K$ to be chosen, we again have the estimate \eqref{eq:dist^2ub}: 
 	\begin{equation}\label{eq:dist^2ub2}
 		\e[\dist^2(0, X_t)] \leq C\sum_{k=1}^K 4^k \p (0 \stackrel{\HH}{\longleftrightarrow} \partial B_{2^{k-1}}) +\e\left[\dist^2(0, X_t) \mathbf{1} \{0 \stackrel{\HH}{\longleftrightarrow} \partial B_{2^{K}}\} \right],\ \ \forall t\geq 0.
 	\end{equation}
 	If we choose $K\in\mathbb{N}\cup\{0\}$ such that $2^{-2(K+1)}<p_c(\mathbb{Z}^d)\mu t\leq 2^{-2K}$ (so we implicitly assume that $\mu t\leq 2$), then \eqref{eq:ppert} and Lemma \ref{lem:onearmhd} imply that
 	\begin{equation*}
 		C_0 2^{-2k} \leq \p (0 \stackrel{\HH}{\longleftrightarrow} \partial B_{2^{k}}) \leq C_1 2^{-2k},\ \ \forall 0\leq k\leq K.
 	\end{equation*}

Setting $T:=\mu ^{-1/2}$, we obtain that $K$ is of order $\log (1/\mu)$, and we also have 
\begin{equation}
    |\log \p (0 \stackrel{\HH}{\longleftrightarrow} \partial B_{2^{K}})| \asymp \log t,~\forall T\leq t \leq 2T.
\end{equation}
So for such $t\in[T,2T]$, the dominant term for the upper bound of the expectation in the RHS of \eqref{eq:dist^2ub2} from Proposition \ref{prop:tlogt} is  the single $\log$ term. Combining this with \eqref{eq:dist^2ub2}, we get
\begin{align}
    \e[\dist^2(0, X_t)] \leq C_2 K+C_3 (t \log t) \mu t \leq C_5 t\mu^{1/2}\log (1/\mu),\ \ \forall T\leq t \leq 2T.
\end{align}
 The rest of the proof is the same as that of Proposition \ref{prop:ublowd}.
 \end{proof}

%\newpage
\section{Supercritical case: proof of Theorem \ref{t:lb}} \label{sec:sup}
Our proof strategy is similar to that of \cite{PSS20}, but some new ideas are needed to get rid of various logarithmic dependencies in \cite{PSS20}. In the next subsection, we introduce the main tools we are going to use: the evolving set process developed by Morris and Peres \cite{MP05} and a coupling of the Markov chain with the Doob transform of the evolving set by Diaconis and Fill \cite{DF90}.

\subsection{Evolving sets and Diaconis-Fill coupling}

We consider the random walk on dynamical percolation in a connected infinite graph $G=(V,E)$. In our applications, we always consider $G=(\mathbb{Z}^d, E(\mathbb{Z}^d))$, but the evolving sets can be defined on arbitrary graphs. We denote the entire evolution of the environment by $\bm{\eta}=(\eta_t: t\ge 0)$, and let $\mathbb{P}^{\bm{\eta}}$ stand for the quenched probability conditioned on $\bm{\eta}$. Then, we discretize time by looking at the random walk at times $n\in \mathbb N$. We consider the time-inhomogeneous Markov chain with transition probability
\begin{equation}\label{eq:discre_MC}
	P_{n+1}^{\bm{\eta}}(x, y):=\mathbb{P}^{\bm{\eta}}\left(X_{n+1}=y \;|\; X_{n}=x\right),\quad \forall x, y \in V, ~  n \in \mathbb{N} \cup\{0\} .
\end{equation}
It is easy to see that $\pi(x)\equiv 1, ~{x \in V},$ is a stationary measure for each $P_{n}^{\bm{\eta}}$. Since the random walk attempts to jump at rate 1, we have
\[P_{n}^{\bm{\eta}}(x, x) \geq \ee^{-1},\quad \forall x \in V, ~  n \in \mathbb{N}.\]
The evolving set process is a Markov chain taking values in the set of all finite subsets of $V$, and its transition is defined as follows: if the current state is $S_n=S\subset V$, then we choose a random variable $U_{n+1}$ uniformly distributed in $[0,1]$, and the next state of the chain is the set 
\[S_{n+1}:=\left\{y\in V: \sum_{x\in S}P_{n+1}^{\bm{\eta}}(x, y)\ge U_{n+1}\right\}.\]
The evolving set process has an absorbing state: $\emptyset$.
Let $K_{P}$ be the transition probability for the evolving set process $(S_n: n\ge 0)$ when the transition matrix for the Markov chain is $P\in\{P_n^{\bm{\eta}}: n\in\mathbb{N}\}$. The Doob transform of the evolving set process conditioned to stay nonempty is defined by the transition kernel
\[\widehat{K}_{P}(A, B):=\frac{|B|}{|A|} K_{P}(A, B).\]
See \cite{MP05}, \cite[Section 6.7]{LP16}, and \cite[Section 17.4]{LP17} for more about evolving sets.

We next define the Diaconis-Fill coupling, which is a coupling between the Markov chain $(X_n: n\ge 0)$ and the Doob transform of the evolving set process. Let $\mathrm{DF}:=\{(x, A): x\in A \text{ and } A\subset V \text{ is finite}\}$. The Diaconis-Fill transition kernel on $\mathrm{DF}$ is defined as 
\[
\widehat{P}^{\bm{\eta}}_{n+1}((x, A), (y, B)):=\frac{P_{n+1}^{\bm{\eta}}(x,y) K_{P_{n+1}^{\bm{\eta}}}(A, B)}{\sum_{z\in A}P_{n+1}^{\bm{\eta}}(z,y)},\quad \forall (x, A), (y, B)\in\mathrm{DF}.
\]
%defines a transition kernel on $\mathrm{DF}$. 
Let $((X_n, S_n): n\ge 0)$ be a Markov chain with initial state $(x, \{x\})\in \mathrm{DF}$ and transition kernel $\widehat{P}^{\bm{\eta}}_{n+1}$ from time $n$ to $n+1$. This coupling, denoted by $\widehat{\mathbb{P}}^{\bm{\eta}}$, has the following properties:
\begin{itemize}
	\item The marginal distribution of $\widehat{\mathbb{P}}^{\bm{\eta}}$ on the first component gives the law of the chain $(X_n: n\ge 0)$ with transition kernel $\{P_{n+1}^{\bm{\eta}}:n\geq 0\}$.
	\item The marginal distribution of $\widehat{\mathbb{P}}^{\bm{\eta}}$ on the second component gives the law of the chain $(S_n: n\ge 0)$ with transition kernel $\{\widehat{K}_{P_{n+1}}: n\geq 0\}$.
	\item For each $y\in S_n$, we have 
	\begin{equation} \label{eq:DF_key}
		\widehat{\mathbb{P}}^{\bm{\eta}}(X_n=y\;|\; S_0, S_1, \ldots , S_n)= {|S_n|}^{-1}.
	\end{equation} 
\end{itemize}
See Theorem 17.23 in \cite{LP17} for a proof of these properties.

We denote by \smash{$\widehat{\mathbb{P}}^{\bm{\eta}}$} the probability measure arising from the Diaconis-Fill coupling with initial state $(0,\{0\})$ (where we assume $0\in V$ and it is the origin when $V=\mathbb{Z}^d$) and the whole evolution of the environment given by $\bm{\eta}$.
We use \smash{$\widehat{\mathbb{P}}$} to denote the averaged probability measure with respect to $\bm{\eta}$ when the initial bond configuration is given by $\pi_p$. 
Furthermore, let $\mathcal{P}$ be the probability measure of the environment when the initial environment is $\pi_p$.  We write $\widehat{\mathbb{E}}^{\bm{\eta}}$, $\widehat{\mathbb{E}}$, and $\mathcal{E}$ for the corresponding expectations.

% we write \smash{$\widehat{\mathbb{P}}^{\bm{\eta}}$} (respectively, $\widehat{\mathbb{P}}$) for the probability measure arising from the Diaconis-Fill coupling with initial state $(0,\{0\})$ (where we assume $0\in V$ and it is the origin if $V=\mathbb{Z}^d$) when the whole evolution of the environment is given by $\bm{\eta}$ (respectively, underlying initial bond configuration is the stationary distribution $\pi_p=\mathrm{Ber}(p)^{E}$). Let $\mathcal{P}$ be the probability measure of the environment when the starting environment is $\pi_p$.  We write $\widehat{\mathbb{E}}^{\bm{\eta}}$, $\widehat{\mathbb{E}}$ and $\mathcal{E}$ for the corresponding expectations. 

For each subgraph $H$ of $G$ and $S\subset V$, we denote by $\partial_{H} S$ the \textbf{edge boundary} of $S$ in $H$, i.e., the set of edges in $E(H)$ with one endpoint in $S$ and the other endpoint in $V \setminus S$. We will also view $\eta_{t}$ as a subgraph of $G$ with vertex set $V$. We now recall an important property about evolving sets from \cite{GJPSWY24}. 

\begin{lemma}[Lemma 4.1 of \cite{GJPSWY24}]\label{lem:evo}
	Consider the random walk on dynamical percolation on $\mathbb{Z}^d$. For each fixed environment ${\bm \eta}$, we have
	\begin{equation}\label{eq:Sn+1}
		\widehat{\mathbb{E}}^{\bm{\eta}}\left[\left.|S_{n+1}|^{-1/2}\;\right|\;S_n\right]\leq \left(1-{\Phi_{S_n}^2}/{6}\right)|S_n|^{-1/2},\quad \forall n\in\mathbb{N}\cup\{0\},
	\end{equation}
   where
	\begin{align}\label{eq.defSn}
		\Phi_{S_{n}} & :=\frac{1}{\left|S_{n}\right|} \sum_{x \in S_{n}} \sum_{y \in S_{n}^{c}} \widehat{\mathbb{P}}^{\bm{\eta}}\left(X_{n+1}=y | X_{n}=x\right),
	\end{align}
	and it satisfies
	\begin{align}\label{eq:Phietan}
		\Phi_{S_n} \geq \frac{1}{ 2d\ee \left|S_{n}\right|} \int_{n}^{n+1}\left|\partial_{\eta_{t}} S_{n}\right| \d t .
	\end{align}
\end{lemma}

\subsection{Good and excellent times}

Let $\theta(p)$ be the percolation probability on $\mathbb{Z}^d$, i.e., the probability that the origin $0$ belongs to an infinite cluster. For $p>p_c(\mathbb{Z}^d)$ and each $t\ge 0$, let $\mathscr{C}_t$ denote the infinite cluster of the dynamical percolation process $\eta_t$.
For every $m\in\mathbb{N}$, we say the time $m$ is {\it good} if 
\begin{align}
 |S_m\cap \mathscr{C}_m|\geq \frac{1}{2}\theta(p)|S_m|.   
\end{align}
A \emph{good time} $m$ is said to be {\it excellent} if the following inequality holds at time $m$:
\begin{equation}
	\int_m^{m+1} |\partial_{\eta_t}S_m| \d t:=\sum_{x\in S_m}\sum_{y\in S_m^c} \int_m^{m+1} \eta_t(x,y) \d t \geq \frac{1}{2}|\partial_{\eta_m} S_m|.
\end{equation}
We will prove that with positive probability, a positive fraction of the time within every fixed time interval is excellent (and thus also good). To begin, we state and prove a simple yet useful lemma.
\begin{lemma}\label{lem:key}
	Suppose $Y$ is a random variable satisfying $Y\leq c$ and $\mathbb{E} [Y]\geq c_1$ for some $c, c_1\in (0,\infty)$. Then
	\begin{equation}
		\mathbb{P}\left(Y > \frac{c_1}{2}\right)\geq \frac{c_1}{2c}.
	\end{equation}
\end{lemma}
\begin{proof}
	Note that
	\[c_1\leq \mathbb{E} [Y]=\mathbb{E}\left[Y\mathbf{1}\{Y>c_1/2\}\right]+\mathbb{E}\left[Y\mathbf{1}\{Y \leq c_1/2\}\right]\leq c\mathbb{P}\left(Y > {c_1}/{2}\right)+c_1/2,\]
	which immediately yields the lemma.
\end{proof}

\begin{lemma}\label{lem:good}
In the setting of \Cref{t:lb}, we have that
	\begin{align}\label{eqn::goodtimepproba}
		&\widehat{\mathbb{P}}\left(m \text{ is good}\right) \geq \frac{\theta (p)}{2},\ \ \forall m\in\mathbb{N},\\
		&\widehat{\mathbb{P}}\left(\# \text{ of good times in }[k,l)>\frac{\theta(p)}{4}(l-k)\right)\geq \frac{\theta(p)}{4},\ \ \forall 1\leq k<l\in \mathbb{N}.\label{eqn::goodtimeexpectation}
	\end{align}
\end{lemma}
\begin{proof}
	Let $Y_m:=|S_m\cap \mathscr{C}_m|/|S_m|$. It is clear that $|Y_m|\leq 1$. Using \eqref{eq:DF_key}, we obtain that
	\begin{align*}
		\widehat{\mathbb{E}}\left[Y_m\right]=\widehat{\mathbb{E}}\left[\frac{|S_m \cap \mathscr{C}_m|}{|S_m|}\right]&=\widehat{\mathbb{E}}\left[\frac{\sum_{x\in S_m}\mathbf{1}\{x\in \mathscr{C}_m\}}{|S_m|}\right]=\int \widehat{\mathbb{E}}^{\bm \eta}\left[\frac{\sum_{x\in S_m}\mathbf{1}\{x\in \mathscr{C}_m\}}{|S_m|}\right] \mathcal{P} (\d {\bm \eta})\\
		&=\int \widehat{\mathbb{E}}^{\bm \eta}\left[\widehat{\mathbb{E}}^{\bm \eta}\left[\frac{\sum_{x\in S_m}\mathbf{1}\{x\in \mathscr{C}_m\}}{|S_m|} \middle | S_m\right]\right] \mathcal{P} (d {\bm \eta})\\
		&=\int \widehat{\mathbb{E}}^{\bm \eta}\left[\widehat{\mathbb{E}}^{\bm \eta}\left[X_m\in \mathscr{C}_m \middle | S_m\right]\right] \mathcal{P} (d {\bm \eta})\\
  &=\int \widehat{\mathbb{E}}^{\bm \eta}\left[X_m\in \mathscr{C}_m \right] \mathcal{P} (d {\bm \eta})=\widehat{\mathbb{P}}\left(X_m\in \mathscr{C}_m \right) =\theta(p),
	\end{align*}
where the last equality follows from the stationarity of the environment seen from the particle. (See \cite[Lemma 2.5]{GJPSWY24} for a proof of this fact in a more general setting.) This completes the proof of~\eqref{eqn::goodtimepproba} by applying Lemma \ref{lem:key}:
 \[\widehat{\mathbb{P}}\left(m \text{ is good}\right)=\widehat{\mathbb{P}}\left(Y_m\ge \frac{\theta(p)}{2}\right)\ge\frac{\theta(p)}{2}.\]

 As a consequence of~\eqref{eqn::goodtimepproba}, we have
\[\widehat{\mathbb{E}}\left[\#\text{ of good times in }[k,l)\right]\ge \frac{\theta(p)}{2}(l-k).\]
This completes the proof of~\eqref{eqn::goodtimeexpectation} by another application of Lemma \ref{lem:key}. 
\end{proof}

\begin{lemma}\label{lem:excellent}
	In the setting of \Cref{t:lb}, we have that
	\begin{equation}\label{eqn::excetimeexpectation}
		\widehat{\mathbb{P}}\left(\#\text{ of excellent times in }[k,l)>\frac{\theta(p)}{8}(l-k)\right)\geq \frac{\theta(p)}{8},\ \ \forall 1\leq k<l\in \mathbb{N}.
	\end{equation}
\end{lemma}
\begin{proof}
Let $\widehat{\mathbb{P}}^{ \eta_0}$ denote the probability measure arising from the Diaconis-Fill coupling with initial state $(0,\{0\})$ when the initial bond configuration of the environment is $\eta_0$. Note that this is different from the quenched measure $\widehat{\mathbb{P}}^{\bm \eta}$. Let $\mathcal{F}_n$ be the $\sigma$-algebra generated by the walk $X_t$, the environment $\eta_t$, and the evolving set $S_t$ up to time $n$. For each $m\in[k,l)\cap\mathbb{Z}$, on the event $\{m \text{ is good}\}$, we have
 \begin{equation}\begin{split}
&~\widehat{\mathbb{P}}^{\eta_0}\left( m \text{ is excellent }|\mathcal{F}_m\right)\\
\geq &~\widehat{\mathbb{P}}^{\eta_0}\left(\text{at least half of the open edges in } \partial_{\eta_m}S_m \text{ do not refresh during } [m,m+1]|\mathcal{F}_m\right).
\end{split}
	\end{equation}
For each edge $e\in \partial_{\eta_m}S_m$, it is clear that 
	\begin{equation}
		\widehat{\mathbb{P}}^{\eta_0}\left(e  \text{ does not refresh during } [m,m+1]| \mathcal{F}_m\right)=\ee^{-\mu}>1/2 \text{ if } \mu\leq 1/\ee.
	\end{equation}
Hence, by the symmetry of the binomial distribution with success probability $1/2$, we have that on the event $\{m \text{ is good}\}$,
	\begin{equation}
		\widehat{\mathbb{P}}^{\eta_0}\left( m \text{ is excellent }|\mathcal{F}_m\right)\geq 1/2.
	\end{equation}
	Therefore,
	\begin{align*}
		\widehat{\mathbb{P}}\left( m \text{ is excellent}\right)&=\int \widehat{\mathbb{P}}^{\eta_0}\left( m \text{ is excellent}\right) \pi_p (\d {\eta_0})=\int  \widehat{\mathbb{E}}^{\eta_0} \left[\widehat{\mathbb{P}}^{\eta_0}\left( m \text{ is excellent} | \mathcal{F}_m\right)\right] \pi_p (\d {\eta_0})\\
		&\geq \frac{1}{2} \int \widehat{\mathbb{P}}^{\eta_0}(m \text{ is good})  \pi_p (\d {\eta_0})=\frac{1}{2}\widehat{\mathbb{P}}(m \text{ is good}) \geq \frac{\theta(p)}{4},
	\end{align*}
	where we used \eqref{eqn::goodtimepproba} in the last inequality. Then, as in the proof of \eqref{eqn::goodtimeexpectation}, the inequatlity \eqref{eqn::excetimeexpectation} follows from another application of Lemma \ref{lem:key}.
\end{proof}

\begin{lemma}\label{lem:excellent1}
In the setting of \Cref{t:lb}, for all $1\leq k<l\in \mathbb{N}$, we have that
	\begin{equation}
		\mathcal{P}\left(\left\{{\bm \eta}: \widehat{\mathbb{P}}^{\bm \eta}\left(\# \text{ of excellent times in }[k,l)>\frac{\theta(p)}{8}(l-k)\right)> \frac{\theta(p)}{16}\right\}\right)\geq \frac{\theta(p)}{16}.
	\end{equation}
\end{lemma}
\begin{proof}
	Define 
	\[Y({\bm \eta}):=\widehat{\mathbb{P}}^{\bm \eta}\left(\#\text{ of excellent times in }[k,l)>\frac{\theta(p)}{8}(l-k)\right).\]
	Clearly, $Y\leq 1$, and by Lemma \ref{lem:excellent}, we have
	\[\mathcal{E}(Y)=\int Y({\bm \eta})\mathcal{P} (\d {\bm \eta})= \widehat{\mathbb{P}}\left(\#\text{ of excellent times in }[k,l)>\frac{\theta(p)}{8}(l-k)\right)\geq \frac{\theta(p)}{8}. \]
	The lemma then follows by applying Lemma \ref{lem:key}.
\end{proof}

We will need the following result about isoperimetric profile by Pete \cite{Pet08}.
\begin{theorem}[Theorem 1.2 of \cite{Pet08}]\label{thm:Pet}
	Fix $d\geq 2$ and $p>p_c(\mathbb{Z}^d)$. There exist constants $\alpha=\alpha(d,p)\in (0,\infty)$ and $C=C(d,p)\in(0,\infty)$ such that for the infinite cluster $\mathscr{C}$, we have that for each $M\in \mathbb{N}$, 
	\[\mathbb{P}\left(\forall S \subset \mathscr{C} \text{ with } 0\in S, S \text{ connected}, M\leq |S|<\infty, \text{ we have }\frac{|\partial_{\mathscr{C}}S|}{ |S|^{1-1/d}}\geq \alpha \right)\geq 1-\exp\left(-CM^{1-1/d}\right).\]
\end{theorem}

A consequence of this theorem is the following corollary.
\begin{corollary}\label{cor:isoperimetric}
	Fix $d\geq 2$, $p>p_c(\mathbb{Z}^d)$, and $c\in(0,\infty)$. There exists $\alpha=\alpha(d,p,c)\in(0,1]$ such that for all sufficiently large $n$, we have
	\[\mathbb{P}\left(\forall S \subset \mathscr{C} \cap B_{n^2} \text{ with }|S| \geq \frac{\theta(p)(\log n)^{3d}}{2}, \text{ we have }\frac{|\partial_{\mathscr{C}}S|}{ |S|^{1-1/d}}\geq \alpha\right)\geq 1-\frac{1}{n^c}.\]
\end{corollary}
\begin{proof}
	If $0\in S$ and $S$ is connected, Theorem \ref{thm:Pet} implies that
	\begin{align}\label{eq:Sconnected}
		&\mathbb{P}\left(\forall S \subset \mathscr{C} \cap B_{n^2} \text{ with } 0\in S, S \text{ connected}, |S| \geq \frac{\theta(p)(\log n)^{3d}}{2}, \text{ we have }\frac{|\partial_{\mathscr{C}}S|}{ |S|^{1-1/d}}\geq \alpha\right)\nonumber\\
		&\quad\geq 1-\exp\left(-C_1(\log n)^{3}\right).
	\end{align}
	
Now, assume that $0\in S\subset \mathscr{C}\cap B_{n^2}$, $S$ is disconnected, and $|S| \geq \theta(p)(\log n)^{3d}/2$. Let $S=S_1\cup\dots\cup S_k$ be the decomposition of $S$ into connected components. Let $A$ be the event
	\[A:=\left\{\forall \tilde{S} \subset \mathscr{C} \cap B_{n^2} \text{ with } \tilde{S} \text{ connected}, |\tilde{S} | \geq \left(\frac{\theta(p)}{2}\right)^{1/d}(\log n)^{3}, \text{ we have }\frac{|\partial_{\mathscr{C}}\tilde{S} |}{ |\tilde{S} |^{1-1/d}}\geq \alpha\right\}.\]
	Then, applying Theorem \ref{thm:Pet} and using a union bound over $x\in B_{n^2}$, we obtain that
% Theorem \ref{thm:Pet} and a union bound imply that
	\begin{align}\label{eq:tildeSconnected}
 \mathbb{P}(A)\geq 1-\exp\left(-C_2(\log n)^{3/2}\right).\end{align}
	We call $S_i$ {\it big} if $|S_i|\geq (\theta(p)/2)^{1/d}(\log n)^{3}$; otherwise, we call $S_i$ {\it small}.

 \medskip
	\textit{ Case I: $\sum_{i: S_i \text{ small}}|S_i|>{|S|}/{2}$.} If $S_i$ is small, then $|S_i|<|S|^{1/d}$ since we assume $|S|^{1/d}\geq (\theta(p)/2)^{1/d}(\log n)^{3}$. Therefore, we have
	\[\#\{i: S_i\text{ small}\} \geq \frac{|S|^{1-1/d}}{2},\]
	which implies that
	\[|\partial_{\mathscr{C}}S| = \sum_{i=1}^k|\partial_{\mathscr{C}}S_i| \geq \sum_{i: S_i \text{ small}}|\partial_{\mathscr{C}}S_i| \geq \sum_{i: S_i \text{ small}} 1 \geq \frac{|S|^{1-1/d}}{2}.\]

 \medskip
	\textit{Case II: $\sum_{i: S_i \text{ big}}|S_i| \geq {|S|}/{2}$.} Conditioning on the event $A$, for each big $S_i$, we have $|\partial_{\mathscr{C}}S_i|>\alpha |S_i|^{1-1/d}$. Therefore,
	\[|\partial_{\mathscr{C}}S| = \sum_{i=1}^k|\partial_{\mathscr{C}}S_i| \geq \sum_{i: S_i \text{ big}}|\partial_{\mathscr{C}}S_i| \geq \alpha |S|^{-1/d}\sum_{i: S_i \text{ big}} |S_i| \geq \frac{\alpha|S|^{1-1/d}}{2}.\]

\medskip
 
	Combining \eqref{eq:Sconnected}, \eqref{eq:tildeSconnected}, and the above Cases I and II, we obtain that
		\begin{align*}
		&\mathbb{P}\left(\forall S \subset \mathscr{C} \cap B_{n^2} \text{ with } 0\in S, |S| \geq \frac{\theta(p)(\log n)^{3d}}{2}, \text{ we have }\frac{|\partial_{\mathscr{C}}S|}{ |S|^{1-1/d}}\geq \frac{\alpha}{2}\right)\nonumber\\
		&\quad\geq 1-\exp\left(-C_3(\log n)^{3/2}\right).
	\end{align*}
	The Corollary then follows by a union bound over all $x\in B_{n^2}$.
\end{proof}

We choose $S\subset B_{n^2}$ in Corollary \ref{cor:isoperimetric} because the evolving set $S_t$ for all $t\in[1,\lambda n]$ (for some fixed $\lambda$) belongs to $B_{n^2}$ with high probability.
\begin{lemma}\label{lem:B_n^2}
	Fix $\lambda \in (0,\infty)$ and the environment ${\bm \eta}$. 
 There exists a constant $C=C(\lambda)\in(0,\infty)$, independent of ${\bm \eta}$, such that
	\[\widehat{\mathbb{P}}^{\bm \eta}\left(\bigcap_{k=0}^{\lceil \lambda n\rceil}\{S_k\subset B_{n^2}\}\right)\geq 1-C\exp[-n^2],~\forall n\in\mathbb{N}.\]
\end{lemma}
\begin{proof}
%Without loss of generality, given a large $n$, we may assume that $\lambda n\in \mathbb{N}$. 
	It is clear that
	\[\widehat{\mathbb{P}}^{\bm \eta}(X_k\notin B_{n^2})\leq \mathbb{P}(\text{Poisson}(k)>n^2)\leq \left(\frac{\ee k}{n^2}\right)^{n^2} \ee^{-k},~\forall k\in[0,\lambda n],\]
	where $\text{Poisson}(k)$ is a Poisson random variable with parameter $k$, and the last inequality follows from its tail probabilities. % of such a random variable.	
	By \eqref{eq:DF_key}, we have
 \begin{equation*}
     \widehat{\mathbb{P}}^{\bm \eta}(X_k\notin B_{n^2}| S_k\not\subset B_{n^2})=\frac{|S_k\cap B_{n^2}^c|}{|S_k|}=\frac{|S_k\cap B_{n^2}^c|}{|S_k\cap B_{n^2}^c|+|
S_k\cap B_{n^2}|}\geq \frac{1}{(2n^2+1)^{d}+1}.
 \end{equation*}
	Therefore, we have
	\[\widehat{\mathbb{P}}^{\bm \eta}(S_k\not\subset B_{n^2})=\frac{\widehat{\mathbb{P}}^{\bm \eta}(X_k\not\subset B_{n^2}, S_k\not\subset B_{n^2})}{\widehat{\mathbb{P}}^{\bm \eta}(X_k\not\subset B_{n^2}| S_k\not\subset B_{n^2})}\leq [(2n^2+1)^{d}+1]\left(\frac{\ee k}{n^2}\right)^{n^2} \ee^{-k}.\]
	Consequently,
	\[\widehat{\mathbb{P}}^{\bm \eta}\left(\bigcup_{k=0}^{\lceil \lambda n\rceil}\{S_k\not\subset B_{n^2}\}\right)\leq \sum_{k=0}^{\lceil \lambda n\rceil}\widehat{\mathbb{P}}^{\bm \eta}\left(S_k\not\subset B_{n^2}\right)\leq [(2n^2+1)^{d}+1]\left(\ee\frac{\lambda +n^{-1}}{n}\right)^{n^2}\frac{1}{1-\ee^{-1}},\]
	which completes the proof of the lemma.
\end{proof}

Let $t(n):=\lceil 8n/\theta(p) \rceil$, we call ${\bm \eta}$ an {\it $(\alpha, n)$-good} environment if the following two conditions hold:
\begin{enumerate}
	\item for each $m\in[1,t(n)]\cap \mathbb{Z}$ and each $S\subset \mathscr{C}_m\cap B_{n^2}$ with $|S|\geq \theta(p)(\log n)^{3d}/2$, we have $|\partial_{\mathscr{C}_m}S|\geq \alpha |S|^{1-1/d}$;
	\item we have
	\[\widehat{\mathbb{P}}^{\bm \eta}\left(\#\text{ of excellent times in }[1,t(n)]>n\right)> \frac{\theta(p)}{16}.\]
\end{enumerate} 
The next lemma says that an $(\alpha, n)$-good environment happens with positive probability.
\begin{lemma}\label{lem:goodenv}
	For the $\alpha$ defined in Corollary \ref{cor:isoperimetric}, there exists $N_0\in\mathbb{N}$ such that for all $n\geq N_0$,
	\[\mathcal{P}\left({\bm \eta} \text{ is $(\alpha, n)$-good}\right)\geq\frac{\theta(p)}{32}.\]
\end{lemma}
\begin{proof}
	This follows directly from Lemma \ref{lem:excellent1} and Corollary \ref{cor:isoperimetric}.
\end{proof}

Fix an environment ${\bm \eta}$. Let $\tau_0:=0$ and $\tau_{k+1}$ be the first excellent time after time $\tau_{k}$ for each $k\in\mathbb{N}$. Define
\begin{equation}
	T_n:=\inf\{k\in\mathbb{N}: S_k\not\subset B_{n^2}\}.
\end{equation}
Note that in the definitions of $\tau_i$ and $T_n$, we have suppressed the dependence on ${\bm \eta}$. The following lemma deals with the drift of $|S_t|^{-1/2}$ in an $(\alpha, n)$-good environment.
\begin{lemma}\label{lem:drift}
	Let ${\bm \eta}$ be an $(\alpha, n)$-good environment. Then, for each $1\leq i <n$, we have
	\begin{equation}
		\widehat{\mathbb{E}}^{\bm \eta}\left[|S_{\tau_{i+1}}|^{-1/2} \mathbf{1}\{t(n)\wedge T_n >\tau_{i+1}\} \big| \mathcal{F}_{\tau_i}\right]\leq|S_{\tau_{i}}|^{-1/2} \mathbf{1}\{t(n)\wedge T_n >\tau_{i}\}\left(1-\phi^2(|S_{\tau_i}|)\right),
	\end{equation}
	where $\mathcal{F}_t$ is the $\sigma$-algebra generated by the evolving set up to time $t$, and $\phi$ is a function defined as 
	\begin{equation}\label{eq:def_phi}
		\phi(r):=\begin{cases}
		    cr^{-1/d}, & r\geq (\log n)^{3d^2}\\
			c (\log n)^{-3d}, & (\log n)^{3d}\leq r <(\log n)^{3d^2}\\
			c/r, & 1\leq r <(\log n)^{3d}
		\end{cases}
	\end{equation}
	for some constant $c=c(d,p)\in(0,\infty)$.
\end{lemma}
\begin{proof}
	Since $\tau_i$ is a stopping time, we have $\{t(n)\wedge T_n>\tau_i\}\in \mathcal{F}_{\tau_i}$. Thus,
	\begin{equation}\label{eq:Staui+11}
        \begin{split}
		\widehat{\mathbb{E}}^{\bm \eta}\left[|S_{\tau_{i+1}}|^{-1/2} \mathbf{1}\{t(n)\wedge T_n >\tau_{i+1}\} \big| \mathcal{F}_{\tau_i}\right] &\leq \widehat{\mathbb{E}}^{\bm \eta}\left[|S_{\tau_{i+1}}|^{-1/2}\mathbf{1}\{t(n)\wedge T_n >\tau_{i}\} \big| \mathcal{F}_{\tau_i}  \right]\\
        &= \mathbf{1}\{t(n)\wedge T_n >\tau_{i}\} \widehat{\mathbb{E}}^{\bm \eta}\left[|S_{\tau_{i+1}}|^{-1/2} \big| \mathcal{F}_{\tau_i}\right],
        \end{split}
	\end{equation}
 where we use the convention $S_{\infty}=\mathbb{Z}^d$.
%Due to the fact that ${\bm \eta}$ is an $(\alpha, n)$-good environment, $\tau_n<\infty$ $\widehat{\mathbb{P}}^{\bm \eta}$-a.s. 
Lemma \ref{lem:evo} implies that $|S_k|^{-1/2}$ is a $\mathcal{F}_k$-supermartingale, so we have that for each $1\leq i <n$,
	\begin{equation}\label{eq:Staui+12}
		 \mathbf{1} \{t(n) >\tau_{i}\} \widehat{\mathbb{E}}^{\bm \eta}\left[|S_{\tau_{i+1}}|^{-1/2} \big| \mathcal{F}_{\tau_i}\right]\leq  \mathbf{1} \{t(n) >\tau_{i}\}  \widehat{\mathbb{E}}^{\bm \eta}\left[|S_{1+\tau_{i}}|^{-1/2} \big| \mathcal{F}_{\tau_i}\right].
	\end{equation}
	Using the Markov property, we can write that
        \begin{multline}\label{eq:St+11}
            \mathbf{1} \{t(n) >\tau_{i}\} \widehat{\mathbb{E}}^{\bm \eta}\left[|S_{1+\tau_{i}}|^{-1/2} \big| \mathcal{F}_{\tau_i}\right]\\
             =\sum_{1\leq m<t(n),S}\widehat{\mathbb{E}}^{\bm \eta}\left[|S_{m+1}|^{-1/2} \big| \tau_i=m, S_m=S\right] \mathbf{1}\{\tau_i=m, S_m=S\}.
        \end{multline}
	Note that $\{\tau_i=m\}\in \mathcal{F}_m$ and $S_{m+1}$ only depends on $S_m$ and the outcome of the independent uniform random variable $U_{m+1}$. Hence,
	\begin{equation}\label{eq:St+12}
		\widehat{\mathbb{E}}^{\bm \eta}\left[|S_{m+1}|^{-1/2} \big| \tau_i=m, S_m=S\right]=\widehat{\mathbb{E}}^{\bm \eta}\left[|S_{m+1}|^{-1/2} \big| S_m=S\right],\ \ \forall m\in \mathbb N, \text{ finite } S\subset \mathbb Z^d.
	\end{equation}
	Now, Lemma \ref{lem:evo} gives that
	\begin{equation}\label{eq:Phi6}
		\widehat{\mathbb{E}}^{\bm \eta}\left[|S_{m+1}|^{-1/2} \big| S_m\right] \leq \left(1-\Phi_{S_m}^2/6\right) |S_m|^{-1/2}, 
	\end{equation}
	where
	\[\Phi_{S_m} \geq \frac{1}{ 2d \ee \left|S_{m}\right|} \int_{m}^{m+1}\left|\partial_{\eta_{t}} S_{m}\right| \d t.\]
 Combining \eqref{eq:Staui+11}--\eqref{eq:Phi6}, we get
 \begin{equation}\label{eq:superm}
 \begin{split}
     &\widehat{\mathbb{E}}^{\bm \eta}\left[|S_{\tau_{i+1}}|^{-1/2} \mathbf{1}\{t(n)\wedge T_n >\tau_{i+1}\} \big| \mathcal{F}_{\tau_i}\right]\\
     &\qquad \leq \mathbf{1}\{t(n)\wedge T_n >\tau_{i}\}\sum_{1\leq m<t(n),S}\mathbf{1}\{\tau_i=m, S_m=S\}\left(1-\Phi_{S_m}^2/6\right) |S_m|^{-1/2}.
     \end{split}
 \end{equation}
	Since $m$ is an excellent time, we have
	\begin{equation}\label{eq:excellentprop}
		\Phi_{S_m}\geq \frac{1}{4d\ee} \frac{|\partial_{\eta_m}S_m|}{|S_m|},\qquad |S_m \cap \mathscr{C}_m|\geq \frac{\theta(p)}{2}|S_m|.
	\end{equation}
	Using $|\partial_{\eta_m}S_m|\geq |\partial _{\mathscr{C}_m}S_m|=|\partial _{\mathscr{C}_m}(S_m\cap \mathscr{C}_m)|$, we get
	\begin{equation}\label{eq:Philb}
		\Phi_{S_m}\geq \frac{1}{4d\ee} \frac{|\partial _{\mathscr{C}_m}(S_m\cap \mathscr{C}_m)|}{|S_m|}.
	\end{equation}
	There are two cases. 

 \medskip
	Case I. If $1\leq |S_m| < (\log n)^{3d}$, the second inequality in \eqref{eq:excellentprop} and \eqref{eq:Philb} imply that
	\begin{equation}\label{eq:smallS}
		\Phi_{S_m}\geq \frac{1}{4d\ee} \frac{1}{|S_m|}.
	\end{equation}

 \medskip
	Case II. If $|S_m| \geq (\log n)^{3d}$, the second inequality in \eqref{eq:excellentprop}  implies that $|S_m \cap \mathscr{C}_m|\geq \theta(p)(\log n)^{3d}/2$. Since $m\leq t(n) \wedge T_n$, by the definition of $(\alpha, n)$-good environment, \eqref{eq:excellentprop} and \eqref{eq:Philb} together imply that
	\begin{equation}\label{eq:largeS}
		\Phi_{S_m}\geq \frac{1}{4d\ee} \frac{\alpha |S_m\cap \mathscr{C}_m|^{1-1/d}}{|S_m|}\geq \frac{\alpha}{4d\ee} \left(\frac{\theta(p)}{2}\right)^{1-1/d}|S_m|^{-1/d}=:c_1|S_m|^{-1/d}.
	\end{equation}

  \medskip
	Substituting \eqref{eq:smallS} and \eqref{eq:largeS} into  \eqref{eq:superm}, we get
	\[\widehat{\mathbb{E}}^{\bm \eta}\left[|S_{\tau_{i+1}}|^{-1/2} \mathbf{1}\{t(n)\wedge T_n >\tau_{i+1} \big| \mathcal{F}_{\tau_i}\}\right] \leq |S_{\tau_i}|^{-1/2}\mathbf{1}\{t(n)\wedge T_n >\tau_{i}\}\left(1-\tilde \phi^2(|S_{\tau_i}|)\right),\]
	where 
	\begin{equation}
		\tilde \phi(r):=\begin{cases}
			c_1 r^{-1/d}, & r\geq (\log n)^{3d},\\
			c_2/r, & 1\leq r <(\log n)^{3d}.
		\end{cases}
	\end{equation}
	This completes the proof of the lemma by taking $c:=\min\{c_1,c_2\}$ and using the trivial bound $-r^{-1/d} \le -(\log n)^{-3d}$ for $(\log n)^{3d}\leq r <(\log n)^{3d^2}$.
\end{proof}

\subsection{Proof of Theorem \ref{t:lb}}

The following Proposition says that for an $(\alpha, n)$-good environment, the size of the evolving set at time $n$ is at least $c_1 n^{d/2}$ with positive probability.
\begin{proposition}\label{prop:largeevo}
	Fix $d\geq 2$, $p>p_c(\mathbb{Z}^d)$, and $\alpha$ from Corollary \ref{cor:isoperimetric}. There exist constants $c_1 \in(0,\infty)$, $c_2\in(0,1)$, and $N_0\in\mathbb{N}$ such that for every $n\geq N_0$ and $(\alpha, n)$-good environment ${\bm \eta}$, we have
	\begin{equation}
		\widehat{\mathbb{P}}^{\bm \eta}\left(|S_n|> c_1n^{d/2}\right)>c_2 .
	\end{equation}
\end{proposition}
\begin{proof}
	Define 
	\[Y_i:=|S_{\tau_i}|^{-1/2}\mathbf{1}\{t(n)\wedge T_n >\tau_{i}\},\ \ i=1,2,\dots, n,\]
	and $f_0$ with $\phi$ defined in \eqref{eq:def_phi}
	\begin{equation*}
		f_0(z):=\begin{cases}
			\phi^2(z^{-2}),&z\in(0,1],\\
			0,&z=0.
		\end{cases}
	\end{equation*}
	Suppose ${\bm \eta}$ is an $(\alpha, n)$-good environment. Lemma \ref{lem:drift} gives that
	\[\widehat{\mathbb{E}}^{\bm \eta}[Y_{i+1}|Y_i]\leq Y_i(1-f_0(Y_i)),\ \ i=1,\dots,n.\]
	Note that $Y_i\leq 1$ for each $i$, and $f_0$ is increasing since $\phi$ is decreasing. Then, Lemma 11 (iii) of \cite{MP05} implies that
	\begin{equation}\label{eq:evolvingset}
 \text{if }k\geq \int_{\delta}^1 \frac{1}{zf(z)} \d z \text{ for some }\delta>0, \text{ then }\widehat{\mathbb{E}}^{\bm \eta}[Y_k]\leq \delta, \text{ where }f(z):=\frac{1}{2}f_0(z/2).
 \end{equation}
	A change of variables gives that
	\[\int_{\delta}^1 \frac{1}{zf(z)}\d z= \int_{4}^{4\delta^{-2}} \frac{1}{r\phi^2(r)} \d r.\]
	Plugging in the function $\phi$ defined in Lemma \ref{lem:drift}, we see that there exist constants $c_0\in (0,\infty)$ and $N_0\in\mathbb{N}$ such that for $\delta:=c_0/n^{d/4}$, we have 
	\[\int_{\delta}^1 \frac{1}{zf(z)} \d z \leq n,~\forall n\geq N_0.\]
	Therefore, using \eqref{eq:evolvingset}, we obtain that for every $n\geq N_0$ and $(\alpha, n)$-good environment ${\bm \eta}$,
	\[\widehat{\mathbb{E}}^{\bm \eta}\left[|S_{\tau_n}|^{-1/2}\mathbf{1}\{t(n)\wedge T_n >\tau_{n}\}\right]\leq \frac{c_0}{n^{d/4}}~ .\]
	
	Next, there exists $N_1\in\mathbb{N}$ such that for every $n\ge N_1$ and $(\alpha, n)$-good environment ${\bm \eta}$,
	\begin{align*}
		&\widehat{\mathbb{P}}^{\bm \eta}\left(|S_k|^{-1/2}\geq \frac{32}{\theta(p)}\frac{c_0}{n^{d/4}} \text{ for each }k\leq t(n)\right)\\ &\quad\leq \widehat{\mathbb{P}}^{\bm \eta}\left(|S_{\tau_n}|^{-1/2}\mathbf{1}\{t(n)\wedge T_n >\tau_{n}\}\geq \frac{32}{\theta(p)}\frac{c_0}{n^{d/4}}\right) + \widehat{\mathbb{P}}^{\bm \eta}(t(n)\wedge T_n \leq\tau_{n})\\
		&\quad \leq \frac{\theta(p)}{32}+\widehat{\mathbb{P}}^{\bm \eta}(t(n)\leq \tau_n)+\widehat{\mathbb{P}}^{\bm \eta}(T_n\leq t_n)\\
		&\quad <\frac{\theta(p)}{32}+1-\frac{\theta(p)}{16}+C(8/\theta(p))\ee^{-n^2}\\
		&\quad < 1-\frac{\theta(p)}{64},
  %\ \ \forall n\geq N_1,
	\end{align*}
	where we used Markov's inequality in the second inequality, the definition of $(\alpha, n)$-good environment and Lemma \ref{lem:B_n^2} in the third inequality, and we have chosen $N_1$ so that the last inequality holds.
	To summarize, we have just proved that for every $n\ge N_1$ and $(\alpha, n)$-good environment ${\bm \eta}$,
	\begin{equation}\label{inequality_1}
		\widehat{\mathbb{P}}^{\bm \eta}\left(|S_k|^{-1/2}< \frac{32}{\theta(p)}\frac{c_0}{n^{d/4}} \text{ for some }k\leq t(n)\right)>\frac{\theta(p)}{64}. %,\ \ \forall n\geq N_1.
	\end{equation}
	By Lemma \ref{lem:evo}, $|S_k|^{-1/2}$ is a nonnegative supermartingale. So Dubins' inequality implies that
	\begin{equation}\label{inequality_2}\begin{split}
		&\widehat{\mathbb{P}}^{\bm \eta}\left(|S_k|^{-1/2}< \frac{32}{\theta(p)}\frac{c_0}{n^{d/4}} \text{ for some }k\leq t(n), |S_l|^{-1/2}\geq \frac{64^2}{\theta^2(p)}\frac{c_0}{n^{d/4}} \text{ for some }l> t(n)\right)\\
  &\leq \frac{\theta(p)}{128} . %,\ \ \forall n\geq N_1.
  \end{split}
	\end{equation}
	Combining the above two inequalities \eqref{inequality_1} and \eqref{inequality_2}, we obtain that
	\begin{align*}
		&\widehat{\mathbb{P}}^{\bm \eta}\left(|S_k|^{-1/2}\geq \frac{64^2}{\theta^2(p)}\frac{c_0}{n^{d/4}} \text{ for some }k> t(n)\right) \leq \widehat{\mathbb{P}}^{\bm \eta}\left(|S_k|^{-1/2}\geq \frac{32}{\theta(p)}\frac{c_0}{n^{d/4}} \text{ for each }k\leq t(n)\right)\\
		&\qquad+\widehat{\mathbb{P}}^{\bm \eta}\left(|S_k|^{-1/2}< \frac{32}{\theta(p)}\frac{c_0}{n^{d/4}} \text{ for some }k\leq t(n), |S_l|^{-1/2}\geq \frac{64^2}{\theta^2(p)}\frac{c_0}{n^{d/4}} \text{ for some }l> t(n)\right)\\
		&\quad <1-\frac{\theta(p)}{64}+\frac{\theta(p)}{128}=1-\frac{\theta(p)}{128}. %,\ \ \forall n\geq N_1.
	\end{align*}
	Therefore, we conclude that for every $n\ge N_1$ and $(\alpha, n)$-good environment ${\bm \eta}$,
	\begin{equation}
		\widehat{\mathbb{P}}^{\bm \eta}\left(|S_k|^{-1/2}< \frac{64^2}{\theta^2(p)}\frac{c_0}{n^{d/4}} \text{ for each }k> t(n)\right) >\frac{\theta(p)}{128},%\ \ \forall n\geq N_1,
	\end{equation}
	which finishes the proof of the proposition.
\end{proof}

Finally, we are ready to prove Theorem \ref{t:lb}.
\begin{proof}[Proof of Theorem \ref{t:lb}]
By \eqref{eq:DF_key}, we have that
	\[\widehat{\mathbb{P}}^{\bm \eta}\left(\left. X_n\in B_{c_3n^{1/2}} \right| |S_n|>c_1 n^{d/2}\right)\leq \frac{|B_{c_3n^{1/2}}|}{c_1 n^{d/2}}<\frac{c_2}{2},\ \ \forall n\geq N_2,\]
	where $c_1, c_2$ are the constants from Proposition \ref{prop:largeevo}, and we have chosen $c_3\in(0,\infty)$ and $N_2\in\mathbb{N}$ such that the last inequality hold for all $n\geq N_2$. Combining this with Proposition \ref{prop:largeevo}, we get that for each $ n\geq\max\{N_0,N_2\}$ and $(\alpha, n)$-good environment ${\bm \eta}$ (with $\alpha$ from Corollary \ref{cor:isoperimetric}),
	\begin{equation}
		\widehat{\mathbb{P}}^{\bm \eta}\left(X_n\in B_{c_3n^{1/2}}\right) \leq \widehat{\mathbb{P}}^{\bm \eta}\left(\left. X_n\in B_{c_3n^{1/2}} \right| |S_n|>c_1 n^{d/2}\right)+\widehat{\mathbb{P}}^{\bm \eta}\left(|S_n| \leq c_1 n^{d/2}\right)<1-\frac{c_2}{2}.
	\end{equation}
	Hence, for such $(\alpha, n)$-good environment ${\bm \eta}$, we have
	\begin{equation}
		\widehat{\mathbb{E}}^{\bm \eta} \left[\dist^2(0,X_n)\right] \geq c_2c_3^2n/2. %, \qquad \forall n\geq\max\{N_0,N_2\}.
	\end{equation}
		Now, we apply Lemma \ref{lem:goodenv} to deduce that
	\begin{align} \label{eq:dist^2largen}
		\widehat{\mathbb{E}} \left[\dist^2(0,X_n)\right] &= \int \widehat{\mathbb{E}}^{\bm \eta} \left[\dist^2(0,X_n)\right]  \mathcal{P}(\d {\bm \eta}) \\ \nonumber
        &\geq \int_{\{{\bm \eta} \text{ is } (\alpha,n) \text{-good}\}} \widehat{\mathbb{E}}^{\bm \eta} \left[\dist^2(0,X_n)\right]  \mathcal{P}(\d {\bm \eta})\nonumber\\ \nonumber
		& \geq c_2c_3^2\theta(p)n/64, \qquad \forall n\geq\max\{N_0,N_2\}.
	\end{align}
	
	The diffusion constant defined in \eqref{eq:diffusionconstant}, combined with \eqref{eq:dist^2largen}, implies the existence of a constant $c=c(d,p)\in(0,\infty)$ such that
	\begin{equation}\label{eq:dcld}
		\sigma^2(d,p,\mu) \geq c, ~ \forall \mu\in(0,1/\ee].
	\end{equation}
By \eqref{eq:dist^2tri}, we have that
\begin{equation}\label{eq:dist^2tri'}
    \frac{\widehat{\mathbb{E}}\left[\dist^2(0, X_{ks})\right]}{ks}\leq  C_1 \frac{\widehat{\mathbb{E}}\left[\dist^2(0, X_{s})\right]}{s}, \ \ \forall k\in\mathbb{N}, ~  s > 0.
\end{equation}
(Note that this inequality implicitly implies that $C_1\geq 1$.) 
	Suppose that it were the case that
	\[\inf_{t>0}\frac{\widehat{\mathbb{E}} \left[\dist^2(0,X_t)\right] }{t} <\frac{c}{2C_1}.\]
 Then, we can find $t_0>0$ such that
 \[\frac{\widehat{\mathbb{E}} \left[\dist^2(0,X_{t_0}))\right] }{t_0} < \frac{c}{2C_1}.\]
 Combining this inequality with \eqref{eq:dist^2tri'}, we get
 \[\sigma^2(d,p,\mu)\leq \frac{c}{2},\]
 which contradicts \eqref{eq:dcld}. Therefore, we must have
 \[\inf_{t>0}\frac{\widehat{\mathbb{E}} \left[\dist^2(0,X_t)\right] }{t}\geq\frac{c}{2C_1}, ~ \forall \mu\in(0,1/\ee],\]
which completes the proof of the theorem.
\end{proof}

\section*{Acknowledgments}

Chenlin Gu is supported by the National Key R\&D Program of China (No. 2023YFA1010400) and National Natural Science Foundation of China (12301166).
Jianping Jiang is supported by National Natural Science Foundation of China (12271284 and 12226001). 
Hao Wu is supported by Beijing Natural Science Foundation (JQ20001). 
Fan Yang is supported by the National Key R\&D Program of China (No. 2023YFA1010400). We thank Chendong Song for simulating Figure~\ref{fig.T}. 

\bibliographystyle{abbrv}
\bibliography{ref.bib}

\end{document}